\documentclass[12pt]{amsart}
\pdfoutput=1
\usepackage{amsmath,amsthm,amsfonts,amssymb,times}
\usepackage{latexsym,epsfig,subfigure,afterpage}
\usepackage{hhline}
\usepackage{array}
\usepackage{tabularx}
\usepackage{longtable}
\usepackage{graphics, graphicx}

\DeclareMathAlphabet{\curly}{U}{rsfs}{m}{n}

\theoremstyle{remark}
\newtheorem{remark}{Remark}
\theoremstyle{plain}

\newtheorem{conj}{Conjecture}

\newtheorem{lem}{Lemma}[section]
\newtheorem{thm}{Theorem}
\newtheorem{cor}{Corollary}
\numberwithin{equation}{section}

\newcommand{\ZZ}{{\mathbb Z}}

\newcommand{\NN}{{\mathbb N}}
\renewcommand{\SS}{\mathfrak{S}}  
\newcommand{\PP}{\mathcal{P}}   

\newcommand{\PPP}{\mathbf P}
\newcommand{\EEE}{\mathbf E}

\newcommand{\li}{{\rm li}}

\newcommand{\be}{\begin{equation}}
\newcommand{\ee}{\end{equation}}
\newcommand{\benn}{\begin{equation*}}
\newcommand{\eenn}{\end{equation*}}
\newcommand{\bal}{\begin{align*}}
\newcommand{\ea}{\end{align*}}
\newcommand{\eal}{\ensuremath{\end{align*}}}
\newcommand{\bea}{\begin{eqnarray}}
\newcommand{\eea}{\end{eqnarray}}

\renewcommand{\th}{\ensuremath{\theta}}

\newcommand{\lam}{\ensuremath{\lambda}}

\renewcommand{\a}{\ensuremath{\alpha}}

\renewcommand{\b}{\ensuremath{\beta}}
\newcommand{\del}{\ensuremath{\delta}}
\newcommand{\eps}{\ensuremath{\varepsilon}}
\renewcommand{\(}{\left(}
\renewcommand{\)}{\right)}
\newcommand{\ds}{\displaystyle}
\newcommand{\pfrac}[2]{\left(\frac{#1}{#2}\right)}

\newcommand{\deq}{\overset{d}=}  
\newcommand{\fl}[1]{{\ensuremath{\left\lfloor {#1} \right\rfloor}}}

\newcommand{\xx}{\ensuremath{\mathbf{x}}}

\newcommand{\zz}{\ensuremath{\mathbf{z}}}
\newcommand{\vv}{\ensuremath{\mathbf{v}}}
\newcommand{\mm}{\ensuremath{\mathbf{m}}}

\newcommand{\er}{{\rm e}}  
\newcommand{\order}{\asymp}  

\newcommand{\BB}{\mathcal{B}}
\newcommand{\II}{\mathcal{I}}

\renewcommand{\le}{\leqslant}

\renewcommand{\ge}{\geqslant}

\newcommand{\ms}{\medskip}

\begin{document}

\title{Prime chains and Pratt trees}

\author{Kevin Ford, Sergei V. Konyagin and Florian Luca}
\address{ KF: Department of Mathematics,
University of Illinois at Urbana-Champaign Urbana,
1409 West Green St., Urbana, IL 61801, USA}
\email{ford@math.uiuc.edu}
\address{ SVK: Department of Mechanics and Mathematics,
Moscow State University, Moscow, 119992, Russia}
\email{konyagin@ok.ru}
\address{FL : Instituto de Matem{\'a}ticas,
 Universidad Nacional Autonoma de M{\'e}xico, Ap. Postal 61-3 (Xangari),
C.P. 58089, Morelia, Michoac{\'a}n, M{\'e}xico}
\email{fluca@matmor.unam.mx }

\date{\today}
\begin{abstract}
Prime chains are sequences
$p_1,\ldots,p_k$ of primes for which $p_{j+1}\equiv 1\pmod{p_j}$ for
each $j$.  
We introduce three new methods for counting long prime chains.
The first is used to show that $N(x;p)=O_\eps(x^{1+\eps})$, where
$N(x;p)$ is the number of
chains with $p_1=p$ and $p_k \le px$.  The second method is used to show that
the number of prime chains ending at $p$ is $\order \log p$ for most $p$.
The third method produces the first
nontrivial upper bounds on $H(p)$, the length of the longest chain
with $p_k=p$, valid for almost all $p$.  As a consequence, we also settle a conjecture of Erd\H os, Granville, Pomerance and Spiro from 1990.
A probabilistic model of $H(p)$, based on the theory of branching
random walks, is introduced and analyzed.  The model suggests that
for most $p\le x$, $H(p)$ stays very close to $\er \log\log x$.
\end{abstract}

\thanks{The research of K.~F. was supported in part
by NSF grants DMS-0555367 and DMS-0901339,  
that of S.~K. was
supported in part by Grant 08-01-00208 from the Russian Foundation
for Basic Research, and that of F.~L. was supported in part
by projects PAPIIT 100508 and SEP-CONACyT 79685.
Some of this work was accomplished
while the third author visited the University of Illinois at
Urbana-Champaign in February 2007, supported by NSF grant DMS-0555367.
The first and second authors enjoyed the hospitality of
the Institute for Advanced Study, where some of this work was
done in November 2007 and spring 2010, the first author being supported
by The Ellentuck Fund and The Friends of the Institute for Advanced Study.
The first author was also supported by the NSF sponsored
Workshop in Linear Analysis and
Probability at Texas A\&M University, August 2008.
}

\maketitle


\section{Introduction}


\subsection{} For positive integers $a$ and $b$, write $a
\prec b$ if $b\equiv 1 \!\!\pmod{a}$.  We are
interested in properties of \emph{prime chains} $p_1 \prec p_2 \prec \cdots
\prec p_k$, e.g. $3 \prec 7 \prec 29 \prec 59$.  
Prime chains are multiplicative analogs of the well-studied
additive prime $k$-tuples (sequences $p_1<\cdots<p_k$ of primes with $p_k-p_1$
small).
Important quantities of study
are $N(x)$, the number of prime chains with $p_k\le x$
($k$ variable), $N(x;p)$, the number of prime chains with $p_1=p$ and 
$p_k/p_1 \le x$, $f(p)$, the number of prime chains with $p_k=p$,
 and $H(p)$, the
length of the longest prime chain with $p_k=p$.  Estimates for these quantities
have arisen  in investigations of iterates of Euler's totient function
$\phi(n)$ and Carmichael's function $\lambda(n)$ (e.g.
\cite{BFLPS}, \cite{BS}, \cite{BKW}, \cite{EGPS}, \cite{Lam}, 
\cite{LP}, \cite{MP}), the
value distribution of $\lambda(n)$ \cite{FL},
common values of $\phi(n)$ and the sum-of-divisors function $\sigma(n)$
\cite{phisig}, and the complexity of
primality certificates (\cite{Bay}, \cite{Pr}).

In studying long chains, where the ratios $\log p_{j+1}/\log p_j$ are small
on average,
we require information about the large prime factors of shifted primes $p-1$.
That is, we require good estimates for $\pi(x;q,1)$ when $q$ is large,
where  $\pi(x;q,a)=|\{ p\le x :  p\equiv a\!\!\pmod{q} \}|$.
Progress, however, is hampered by our poor knowledge when $q>\sqrt{x}$.
Let  $\li(x) = \int_2^x dt/\log t$.
The Bombieri-Vinogradov theorem (\cite{Da}, Ch. 28) implies that
\be\label{BVQ} \sum_{m\le Q}
\max_{y\le x}  \left| \pi(y;m,1)- \frac{\li(y)}{\phi(m)}
\right| \ll R,
\ee
with $Q=x^{1/2}(\log x)^{-B}$ and $R=x(\log x)^{-A}$ 
(here $A>0$ is arbitrary and
$B$ depends on $A$).  The corresponding statement with $Q=x^{\th}$ is
not known for any fixed $\th>1/2$, however it is conjectured 
(the Elliott-Halberstam conjecture; abbreviated EH) that 
\eqref{BVQ} holds with $Q=x^\th$ and $R=x(\log x)^{-A}$ for any $\th<1$ and $A>0$.
The one-sided Brun-Titchmarsh inequality 
\be\label{BT}
\pi(x;q,1) \le \frac{2x}{(q-1)\log(x/q)},
\ee
however, is useful in some situations.

If one asks just for the existence of many shifted primes $p-1$ with a large
prime factor, we can do a little bit better than Bombieri-Vinogradov.
Let $P^+(n)$ denote the largest prime factor of $n$, and
let $\theta_0$ be the supremum of real numbers $\th$ so
 that there are infinitely many primes $p\le x$ such that $P^+(p-1) \ge x^\th$.
EH implies $\th_0=1$, and the best unconditional result is due to Baker and 
Harman \cite{BH}, who showed that $\theta_0 \ge 0.677$. 

In this paper, we prove new bounds for $N(x;p)$, $N(x)$, $f(p)$ and $H(p)$.
At the core of our arguments is a kind of duality principle: in a chain
$p_1 \prec \cdots \prec p_k$, there are integers $m_j$ with $p_{j+1}=m_jp_j+1$,
and there is an obvious bijection between the $k$-tuples $(p_1,\ldots,p_k)$ and
$(p_1,m_1,\ldots,m_{k-1})$.  It is often more efficient to focus on properties 
of the latter vector rather than the former.


\subsection{}
We begin with the problem of bounding $N(x;p)$.
By iterating \eqref{BT},
one arrives at a uniform bound (e.g. \cite{EGPS}, Theorem 3.5)
\be\label{EGPS1}
N_k(x;p) \ll \frac{x (c\log_2 x)^{k-1}}{\log x}
\ee
for the number, $N_k(x;p)$, of prime chains of length $k$ with $p_1=p$
 and $p_k/p_1\le x$. Here  $\log_k x$ is the $k$-th iterate 
of the logarithm of $x$, and $c$ is some constant.
 Summing \eqref{EGPS1} over $k\le \frac{\log x}{\log 2}+1$,
one obtains the weak estimate $N(x;p) \ll x^{O(\log_3 x)}.$

\begin{thm}\label{Pxp}
For $p\ge 2$ and $x\ge 20$, we have the effective estimate
$$
N(x;p) \le x \exp \left\{
 \frac{\log x (\log_3 x + O(1))} {\log_2 x} \right\}.
$$
In particular, for every $\eps>0$ there is an effective constant
$C(\eps)$ so that
$N(x;p) \le C(\eps) x^{1+\eps}.$
\end{thm}

Theorem \ref{Pxp} has applications to problems which, at first glance,
have nothing to do with prime chains.  First, it
is a crucial tool in the recent proof by Ford, Luca and Pomerance
\cite{phisig} that the equation $\phi(a)=\sigma(b)$ has infinitely
 many solutions,
settling a well-known 50-year old problem of Erd\H os.
In \cite{FL}, Theorem \ref{Pxp} is used to show that for some
 effective $q_0$,
if $\pi(p^{3a};p^a,1)-\pi(p^{3a};p^{a+1},1)\ge 113 
p^{\frac{7a-3}{4}}/\log (p^{a+1})$
for all prime powers $p^a\in (10^{10},q_0]$,
then for every positive integer $n$, 
there is another positive integer $m$ with
$\lambda(n)=\lambda(m)$.  This nearly settles a conjecture from \cite{BFLPS},
the analog for $\lambda$ of the famous Carmichael Conjecture for $\phi$.

Theorem \ref{Pxp} is nearly best possible, since 
 $N(x;p) \ge N_2(x;p) = \pi(px;p,1)$, which is expected to be $\gg x/(\log px)$
unless $x$ is very small relative to $p$.

\begin{conj}\label{Pxp_conj}
We have $N(x;p) \ll x$.
\end{conj}

Conjecture \ref{Pxp_conj} is easy to prove when $p$ is bounded.  Using 
$f(2)=1$ and the recursive formula
\be\label{frecursive} 
f(p) = 1+ \sum_{q|(p-1)} f(q), 
\ee
we have 
\be\label{fpupper}
 f(p) \le \frac{2\log p}{\log 2} - 1 \qquad (\text{all } p).
\ee
Summing on $p\le x$
using the prime number theorem gives $N(x)\ll x$ and hence
$N(x;p) \ll px$.  Lower bounds on $f(p)$ and $N(x)$ are more difficult,
since $f(p)$ is sometimes very small, e.g.
if $p=1+2^a3^b$ then $f(p)=4$ (it is conjectured that there are infinitely many
such primes).

\begin{thm}\label{Nx}
(i) We have $f(p)\ge 0.378 \log p$ for almost all primes $p$.  
Hence, $N(x)\gg x$.

(ii) For all $x\ge 3$ and any positive integer $h$,
$\ds | \{ p\le x : f(p)=h \}| \le \pfrac{6\log x}{h}^h.$
\end{thm}

In particular, part (ii) implies that
primes with $f(p)=o(\log p)$ are exceptionally rare, the counting function
being $x^{o(1)}$ as $x\to\infty$.  Also, by (i), we have
$N(x;2)=\frac12 N(x) \gg x$. 

It is likely that $f(p)$ has normal order\footnote{
for every $\eps>0$, $|f(p)-c\log p| \le \eps \log p$ for all primes 
but $o(\pi(x))$ exceptions up to $x$.} 
$c\log p$ for some $c$.  A very similar problem was considered
in Section 2 of \cite{EGPS}, namely the behavior of
$I(n)=\min\{j:\phi_j(n)=1\}$, where $\phi_j$ is the $j$-th iterate of $\phi$.
It turns out that
$F(n)=I(n)-\{^{1 \,\,\,\, n \,\, \text{odd}}_{0 \,\,\,\, n \,\, \text{even}}$ 
is completely additive, and
$F(p)=F(p-1) = \sum_{q^a \| p-1} a F(q)$ for odd primes $p$.
This is similar to \eqref{frecursive}, the only difference
being the behavior at proper prime powers, which play an insignificant
role in the arguments in \cite{EGPS}.  Summing \eqref{frecursive} over
primes $p\le x$ gives
\[
N(x) = \pi(x) + \sum_{q\le x/2} f(q) \pi(x;q,1). 
\]
Inserting this relation into the proof of Theorem 2.1 in \cite{EGPS} 
(combine
Lemma 2.4, Corollary 2.5, (2.8) and Theorem 2.1 therein), we obtain
the following.  

\newtheorem*{theoremA}{Theorem A}
\begin{theoremA}\label{EGPSthm}
If \eqref{BVQ} holds with $Q=x^{1-(\log_2 x)^{-1-\del}}$ and 
$R=x(\log x)^{-2}$
for some fixed $\del>0$,  then $N(x)\sim c x$ for some constant
$c>0$ and $f(p)$ has normal order $c \log p$.
\end{theoremA}

Conjecture \ref{Pxp_conj} implies that for all $\eps>0$
and prime $q>(\log x)^{1+\eps}$, for most $p\le x$ there is no
prime chain $q \prec \cdots \prec p$.
 This gives, conditionally, the first part of \cite[Conjecture 1]{EGPS}.
By contrast,
the proof of Theorem 4.5 of \cite{EGPS} implies that if $q\le (\log x)^c$,
for some small constant $c>0$, then  for
almost all primes $p\le x$, there is a prime chain $q\prec \cdots \prec p$.


\subsection{}
The {\it Pratt tree} $T(p)$ for a prime $p$ is the tree with root node $p$,
below $p$ are nodes labelled with the prime factors $q$ of $p-1$, below
each $q$ are nodes labelled with the prime factors of $q-1$, and so on.
In 1975, V. Pratt  \cite{Pr} used
it in conjunction with Lucas' primality test (\cite{CP}, \S 4.1)
to show that every prime has a short certificate (proof of primality).
Pomerance \cite{Pom87} gave another method for producing primality
certificates, but it is an open
problem whether the Pratt certificate has longer complexity for most primes
(see \S 1 of \cite{Pom87}).
Two important statistics of the Pratt tree are the total number of nodes
$f(p)$ and the height $H(p)$, the latter being the length of the longest
prime chain ending at $p$.
It is known (see \cite{BKW}, \cite{Er35}, \cite{EP}) that the number
of primes at a \emph{fixed} level $n$ in the Pratt tree for most $p$ is 
$\sim (\log_2 p)^{n}/n!$.  The idea is that for most
primes $p$, $p-1$ has a multiplicative structure similar to that of
a typical integer of its size; namely, $p-1$ has
about $\log_2 p$ prime factors, uniformly
distributed on a $\log\log$-scale (see \cite{HaTe}, Ch. 1).
This, however, does not give much information about $H(p)$.

\begin{figure}[t]
\vskip -0.2in
\begin{tabular}{cc}
\includegraphics[height=2.5in]{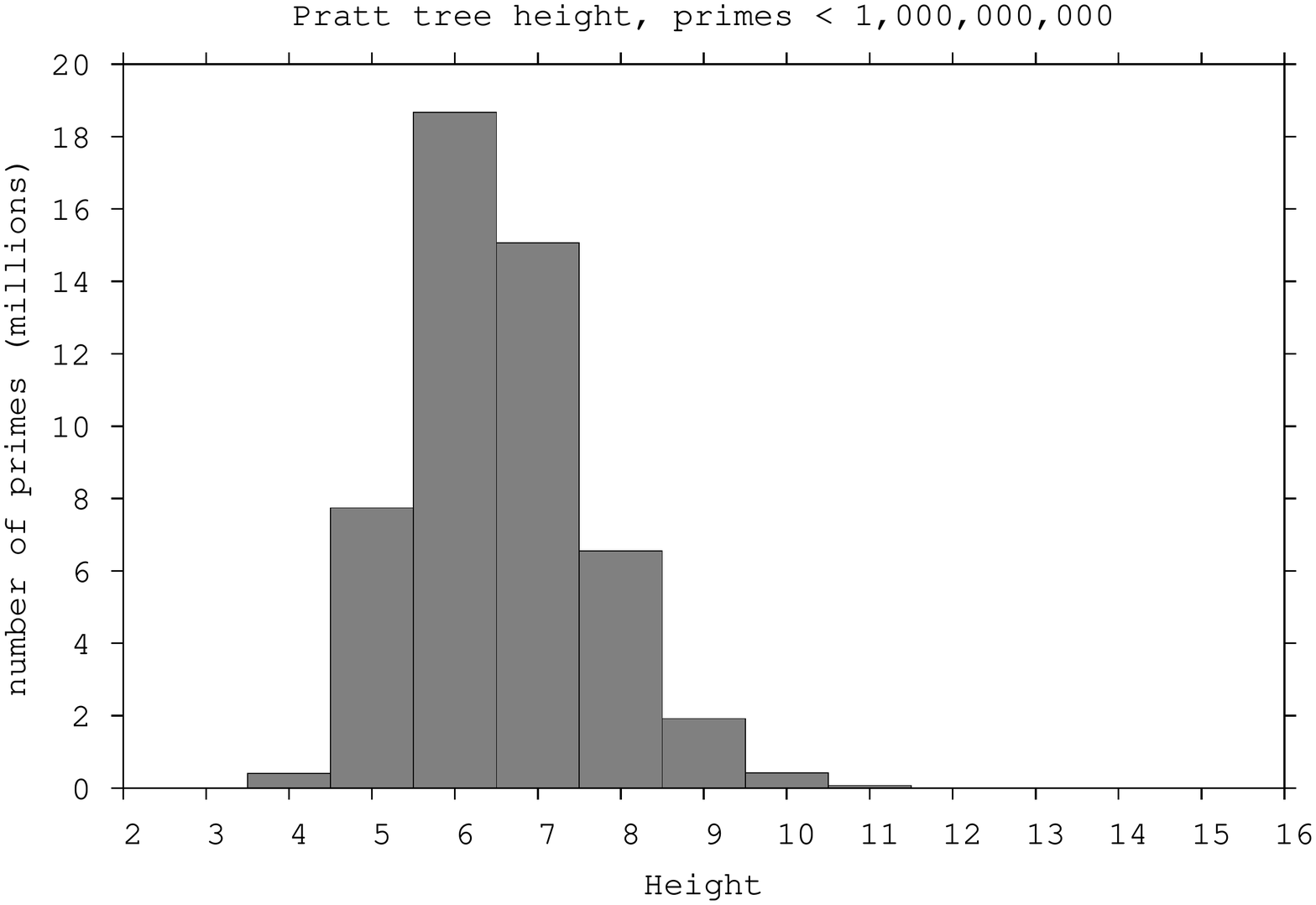} &
\includegraphics[height=2.5in]{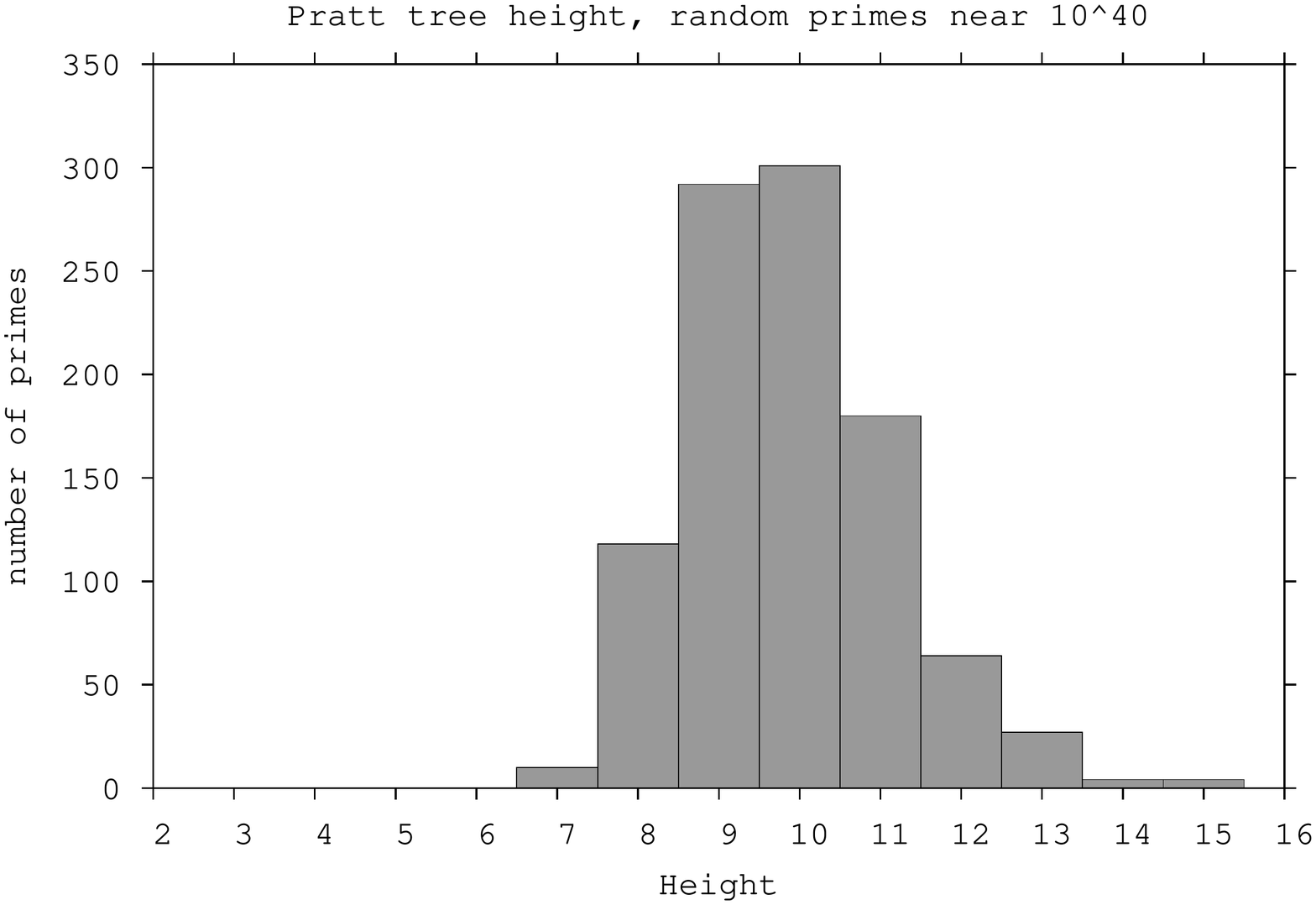} 
\end{tabular}
\vskip -0.3in
\caption{Pratt tree height}\label{figure2}
\end{figure}

Figure 1 shows histograms of $H(p)$ for all primes
$p\le 10^9$ and for 1000 randomly chosen primes near $10^{40}$.
Very little is known about the distribution of $H(p)$,
the extremal behavior being a case in point.
First, $H(p)=2$ if and only if $p$ is a Fermat prime, that is,
$p=2^{2^m}+1$ for some $m$.  It seems plausible that $H(p)=3$ for
infinitely many $p$, but this is hopeless to prove at this time.
At the other extreme, we have the trivial upper bound $H(p)\le \frac{\log p}
{\log 2}+1$.  Large values of $H(p)$ may be obtained using
the special chain $2=q_1 \prec q_2 \prec \cdots$,
where, for each $j$, $q_{j+1}$ is the smallest prime $\equiv 1\pmod{q_j}$.
By Linnik's theorem, $q_{j+1}\le q_j^L$ for some constant $L$, hence
$H(q_j) \ge \frac{\log\log q_j}{\log L}$.  It is conjectured
that $q_{j+1} \le q_j (\log q_j)^C$ for some fixed $C$ and this implies
a far stronger bound $H(q_j) \gg \frac{\log q_j}{\log\log q_j}$. 
Even showing $H(p)/\log_2 p \to \infty$ for an infinite sequence of $p$
is extremely hard, as it implies $\th_0=1$:
a prime chain $p_1 \prec \cdots \prec p_k=p$ with $k=H(p)$ satisfies
\[
 \frac{\log p_k}{\log p_1} = \prod_{j=1}^{k-1}
\frac{\log p_{j+1}}{\log p_j} \ge \( \min_{1\le j\le k-1} 
\frac{\log p_{j+1}}{\log p_j} \)^{k-1},
\]
 and hence 
\be\label{Htheta}
\Lambda := \limsup_{p\to\infty} \frac{H(p)}{\log_2 p} \le \frac{1}{-\log \th_0}.
\ee
On the other hand,
K\'atai \cite{Ka1} proved that for some 
constant $c>0$, $H(p)\ge c\log_2 p$ for almost all primes $p$ (for all primes
$p\le x$ with $o(x/\log x)$ exceptions).
We prove a version with the constant made explicit in terms of 
the level of distribution of primes in progressions.

\begin{thm}\label{Prattlower}
 (a) If \eqref{BVQ} holds with $Q=x^{\th}$ and $R=o(x/\log x)$, then
for any $c<\frac{1}{\er^{-1}-\log\th}$, $H(p)>c\log_2 p$ for almost all 
primes $p$;

(b) If \eqref{BVQ} holds with $Q=x^{\th}$ and $R=x(\log x)^{-A}$
for every $A>1$,
then for every $c<\frac{1}{-\log \th}$, there is a $K$ so that $H(p)>c\log_2 p$ for $\gg x/(\log x)^K$ primes $p\le x$.  Consequently, $\Lambda
\ge \frac{1}{-\log\th}$.
\end{thm}

\begin{cor}\label{HEH}
EH implies that for every $c<\er$, $H(p)>c\log_2 p$ for almost all $p$.
\end{cor}

\begin{remark}
By the Bombieri-Vinogradov theorem and \eqref{BT}, for every $\eps>0$,
\eqref{BVQ} holds with $Q=x^{1-\eps}$ and $R=O_\eps(x/\log x)$.

If $\Lambda  \ge -1/\log \th$ for some $\th<1$, then there are
many chains $2=p_1 \prec \cdots \prec p_k$
with $\frac{\log p_{j+1}}{\log p_j}$ at most $1/\th$ on average.
Thus, although \eqref{Htheta} is likely an equality,
to prove this would require strong information about the set of
primes with $P^+(p-1)$ near $p^{\th_0}$.  The same difficulty
arises when trying to prove (a) with $\th=\th_0$.

The bound on $H(p)$ in (a) is weaker than in (b), since it is unusual in a chain
$p_1\prec \cdots \prec p_k$ for most of  the ratios $\frac{\log p_{j+1}}{\log p_j}$
to be close to $1/\th$.  The constant $\er^{-1}$ appearing in (a) is likely
best possible; see Conjecture \ref{Dpnormal1} below.
\end{remark}

Turning to upper bounds, before our work it was unknown if there 
is an infinite sequence of primes with $H(p)=o(\log p)$.
A natural  approach is to find $p$ such that 
$P^+(p-1)$ is small and use $H(p)=1+\max_{q|(p-1)} H(q)\ll \max_{q|(p-1)}
\log q$.  
However, our knowledge of smooth shifted primes is very weak 
(the world record is $P^+(p-1)<p^{0.2961}$ infinitely often \cite{BH}).
Using a new
and very different approach, we give a much stronger upper bound.

\begin{thm}\label{Prattupper}
We have $H(p) \le (\log p)^{0.9503}$ for almost all $p$.
\end{thm}

The proof of Theorem \ref{Prattupper} involves showing that for most primes
$p$, all the primes at some bounded level of the tree are small.
In particular, this settles \cite[Conjecture 2]{EGPS}.  

\begin{thm}\label{EGPSconj}
For every $\eps>0$ and $\delta>0$, there is an integer $k$ so that for large $x$
and at least $(1-\delta)x$ integers $n\le x$, $P^+(\phi_k(n))\le x^{\eps}$.
\end{thm}

There is a folklore conjecture that $H(p)=O(\log_2 p)$ for most $p$. 

\begin{conj}\label{Dpnormal1}
$H(p)$ has normal order $\er \log_2 p$. 
\end{conj}

\begin{remark}
The lower bound in Conjecture \ref{Dpnormal1} follows from EH (Corollary \ref{HEH}).
Conversely, by \eqref{Htheta}, the lower bound in Conjecture \ref{Dpnormal1}
implies that $\th_0 \ge \er^{-1/\er}=0.6922\ldots$ (with a bit more work
using \eqref{BT}, one can deduce $\th_0\ge 0.73$).  

The upper bound in Conjecture  \ref{Dpnormal1} appears to 
be even more difficult.
We cannot see a way to deduce it from standard conjectures in prime
number theory, e.g. EH plus a uniform prime $k$-tuples conjecture,
although Theorem \ref{Prattupper} can be significantly improved under such hypotheses.
\end{remark}

Understanding $H(p)$ requires detailed knowledge of the distribution
of the large prime factors of shifted primes $p-1$.  Making a reasonable
assumption for this distribution (a consequence of EH),
in Section \ref{sec:model} we model the Pratt tree by a branching random walk.
The model provides a much more precise version of Conjecture \ref{Dpnormal1}.

\begin{conj}\label{Dpnormal2}
$H(p)=\er \log_2 p - \frac32 \log_3 p+E(p)$, where
for  some fixed $c,c'>0$ and  any $z\ge 0$, the number of 
$p\le x$ for which $E(p) \ge z$ is $\gg \er^{-c'z}\pi(x)$ and
$\ll \er^{-cz} \pi(x)$, and $E(p)\le -z$
for $O(\exp\{ - \er^{cz} \} \pi(x))$ primes $\le x$.
\end{conj}

Notable features of Conjecture
\ref{Dpnormal2} are (i) the \emph{tightness} of $E(p)$: the distribution
of $H(p)$ over $p\le x$ does not widen as $x\to\infty$, and (ii) the
pronounced asymmetry of the distribution of $E(p)$.  The analogs of
these features for our probabilistic model are proved rigorously.

Assuming $\text{Median}\{H(p): p\le x\}$ grows slowly, we show that $H(p)$ is tight
to the left of its median.

\begin{thm}\label{tightleft}
Suppose $g$ and $h$ are increasing, $0\le
g(x)\le h(x)$, $h(x^2)-h(x)\le K$ and $g(x^2)-g(x)\le
K$ for $x\ge 1$.  Suppose, for large $x$, that $H(p)\ge h(p)$ for at least
$c\pi(x)$ primes $\le x$.  Then $H(p) \ge h(p) - g(p)$
for all primes $p\le x$ with at most $O(\pi(x) \exp \{ - \frac{c\log
  2}{K} g(x) \})$ exceptions. 
\end{thm} 

We conclude this section with a conjecture about prime chains, which
follows from the prime $k$-tuples conjecture (with $m=2$ below)
but should be ``easier''.
It is a multiplicative analog of the statement that the
primes contain arbitrarily long arithmetic progressions, recently
proved by Green and Tao \cite{GT}.
Even the case $k=3$ is not known.

\begin{conj}
For each $k\ge 3$, there are infinitely many prime $k$-tuples
$(p_1,\ldots, p_k)$ where, for some $m$, $p_{j+1}=m p_j + 1$ for $1\le
j\le k-1$.
\end{conj}

{\bf Notation.} 
The letters $p$ and $q$,
with or without subscripts, always denote primes.  Constants
implied by the $O$, $\ll$ and $\order$ symbols do not depend on any parameter
unless indicated.  In Section \ref{sec:model}, we use $\PPP$ for
probability and $\EEE$ for probabilistic expectation.

%
%
\section{Sifted chains: proof of Theorem \ref{Pxp}}\label{sec:Pxp}
%
%

 The underlying idea is a sieve; relax the
condition that the numbers in the chain are prime, and only
require that they do not have small prime factors. Let $y\ge 2$ and
let $r$ be the product of the primes $\le y$. For $(a,r)=1$, let
$G_a(x;y)$ be the number of chains $n_1 \prec \cdots \prec n_k$ with $n_1=a$,
 with  $n_k/n_1\le x$ and consisting of numbers coprime to $r$. 
If $p>y$, then $N(x;p) \le G_p(x,y)$.
There are integers (``links'') $m_1,
\ldots,m_{k-1}$ with $n_{j+1}=m_jn_j+1$ for $1\le j\le k-1$.
For positive integers $a,b$ and real $s>1$,
$$
S(a,b) = S(a,b;r,s) = \sum_{\substack{m\ge 1 \\ am+1\equiv b(\!\bmod{r})}}
m^{-s}
$$
encodes the possible links $m$ from
a number $n_i \equiv a\pmod{r}$ to a number $n_{i+1} \equiv b
\pmod{r}$.

Fix $r,s$ and let $U_r = (\ZZ/r\ZZ)^*$.
Let $A_k(a_1,a_k)$ be the sum of $(m_1\cdots m_{k-1})^{-s}$ over all
tuples $(m_1,\ldots,m_{k-1})$ which could serve as links in a chain
starting from
a number $n_1\equiv a_1\pmod{r}$, ending with a number $n_k\equiv
a_k\pmod{r}$ and with all numbers in the chain coprime to $r$.
Then $A_2(a_1,a_2)=S(a_1,a_2)$ and for $k\ge 3$,
$$
A_k(a_1,a_k) = \sum_{a_2,\ldots,a_{k-1}\in U_r} S(a_1,a_2) S(a_2,a_3) \cdots
S(a_{k-1},a_k).
$$
Let $V_k(a_1)$ be the column vector $(A_k(a_1,a_k): a_k\in U_r)$.
For consistency, let $V_1(a_1)$ be a vector with all zero entries
except for an entry of 1 in the $a_1$ position.  Since
$$
A_{k+1}(a_1,a_{k+1}) = \sum_{a_k\in U_r} A_k(a_1,a_k) S(a_k,a_{k+1}),
$$
we obtain $V_{k+1}(a_1) = M V_k(a_1),$
where  $M =  M(r,s) = \bigl( S(a,b) \bigr)_{b,a\in U_r}.$
The rows of $M$ are indexed by $b$ and the columns are indexed by
$a$. Finally, let $F_k(a_1) = \sum_{a_k} A_k(a_1,a_k)$, so that
\[
 F_k(a_1) = (1,\ldots,1) V_k(a_1) = (1,\ldots,1)
M^{k-1} V_1(a_1);
\]
i.e., $F_k(a_1)$ is the sum of the entries of
column $a_1$ in $M^{k-1}$.
Since $m_1\cdots m_{k-1} \le n_k/n_1\le x$,
\be\label{NF}
G_a(x;y) \le \inf_{s>1} \Big(x^s \sum_{1\le k\le \frac{\log x}{\log
2}+1} F_k(a)\Big).
\ee
Observe that the sum on $k$ in \eqref{NF}, if extended to
$k=\infty$, is convergent if and only if $M$ is a contracting
matrix, i.e., all the eigenvalues of $M$ have modulus $<1$.  Since
$M$ has positive real entries, the Perron-Frobenius Theorem implies
that the eigenvalue with largest modulus is positive, real and
simple.  Call this eigenvalue $\lam(s;y)$.

We show below that if $y$ is large and $s\ge 1 + \frac{\log_2 y}{\log y}$,
then $\lam(s;y)<1$.
Accurate estimation of
$\lam(s;y)$ is difficult for large $y$, but the largest row sum
of $M$ serves as an upper bound.  For a generic matrix $A$, let
$R_b(A)$ be the sum of the entries in the row indexed by $b$,
and let $R(A)$ be the maximum row sum of $A$.  For row $b$ of $M$,
write $d=(b-1,r)$ and $b'=\frac{b-1}{d}$.  Then
\begin{equation*}
R_b(M) = \sum_{a\in U_r} \sum_{am\equiv b-1 (\!\bmod{r})} m^{-s} =
d^{-s} \sum_{(k,r/d)=1} k^{-s} \# \{ a\in U_r:ak\equiv b'\,
(\!\bmod{\, r/d}) \}.
\end{equation*}
The congruence $ak \equiv b' \pmod{r/d}$
has a unique solution modulo $r/d$, and hence has
$\phi(d)$ solutions $a\in U_r$.  Thus,
\begin{equation*}
R_b(M)=\frac{\phi(d)}{d^s} \sum_{(k,r/d)=1} k^{-s}  =
\frac{\phi(d)}{d^s} \prod_{p\nmid (r/d)} (1-p^{-s})^{-1} = 
\prod_{p>y}(1-p^{-s})^{-1} \prod_{p|d} \frac{p-1}{p^s-1}.
\end{equation*}
Therefore, since $d$ is always even,
\be\label{RM}
R(M) = \frac{1}{2^s-1} \prod_{p>y} (1-p^{-s})^{-1}.
\ee

Since $R(AB)\le R(A) R(B)$, $R(M^{k-1}) \le R(M)^{k-1}$.
To bound $G_a(x;y)$, we need to
bound the largest \emph{column} sum of $M^{k-1}$.  Lacking a better
approach, we use the crude bound $\phi(r) R(M)^{k-1}$.  Thus,
$F_k(a) \le  \phi(q) R(M^{k-1}) \le \phi(r) R(M)^{k-1}.$
By \eqref{NF},
\[
G_a(x;y) \le \phi(r) \inf_{s:R(M)<1}
\frac{x^s}{1-R(M)}.
\]
By standard prime number estimates, if $1<s\le 2$, then
\[
-\sum_{p>y} \log(1-p^{-s}) = O(1/y^{2s-1}) + \sum_{p>y} p^{-s} 
 \ll \frac{\er^{-(s-1)\log y}}{(s-1)\log y}.
\]
Take $y=\frac{\log x}{\log_2 x}$ and $s=1+\frac{\log_2 y}{\log y}$.  
Since $2^s-1=1+(2\log 2)(s-1)+O((s-1)^2)$, 
\eqref{RM} implies
$1-R(M) \sim (2\log 2) (s-1)$ as $x\to\infty$.   
Since $\phi(r) \le r = \er^{(1+o(1))y}$ as $x\to\infty$, this proves 
Theorem \ref{Pxp}.

%
\section{Proof of Theorem \ref{Nx}}\label{sec:Px}
%

Let $Q(p)$ be the multiset of prime labels appearing in $T(p)$,
and let $T'(p)$ be the subtree of $T(p)$ consisting of nodes with
odd prime labels.  There is a natural bijection between $T(p)$ and $T'(p)$,
obtained by adding to every node in $T'(p)$ a child node with label 2.
Let $l(n)=\prod_{p^a \| n} p^{a-1}$.  The quantity $\frac{q}{q-1}l(q-1)$
measures the ``loss of mass'' when descending from a node labelled $q$
to its child nodes: in fact, it is easy to see that
\[
 \prod_{q\in Q(p)} \( \frac{q}{q-1} l(q-1) \) = p.
\]
If $p>2$, exactly half of the nodes in $T(p)$ are labelled with 2, thus
\be\label{prodl}
\prod_{q\in Q(p)} l(q-1) \le p 2^{-\frac12 f(p)}.
\ee

Consider the set of $p\le x$ with $f(p)=h$, where $h=2n$ is positive and even.
Let $\mathcal{T}$ be the set of rooted trees on $n$ nodes.  
For each $T'\in \mathcal{T}$, we
consider separately the primes $p$ with $T'(p)$ tree-isomorphic to $T'$.
Form the tree $T$ on $h$ nodes, by adding to each node of
$T'$ an additional child node.  We count in how many ways we can label with primes the nodes of $T$, with the leafs having label 2 and the root having
label $p\le x$.  For a given prime $p$, there may be more than one way
to assign primes in the Pratt tree to the nodes of $T'$; this occurs
when some node has two or more child trees that are isomorphic (as rooted trees).  Thus, for each $T'$, we count ways to assign primes to the nodes,
and divide by the number of ways in which 
we can permute the nodes; that is, the
number $I(T')$ of isomorphisms of $T'$.
Assign to each node
an ordinal number $1,2,\ldots,h$ so that the children of node $j$
are assigned lower ordinals (e.g., node 1 will be a lowest leaf, and node $h$
will be the root).  To node number $j$, we let $q_j$ be its prime label 
and $l_j=l(q_j-1)$.  Let $B_j$ be the set of ordinal numbers of the 
children of node
$j$, and observe that $B_1,\ldots,B_h$ depend only on $T'$.  With this notation,
\be\label{qlq}
q_j-1 = l_j \prod_{k\in B_j} q_k.
\ee

With $T$ fixed, \eqref{qlq} implies a natural bijection between 
$(q_1,\ldots,q_h)$ and $(l_1,\ldots,l_h)$ (recall that leafs have $q_j=l_j=2$).
By \eqref{prodl}, we have for any $\beta>0$
\[
 |\{ p\le x: f(p)=h \}| \le \sum_{T'\in \mathcal{T}} \frac{1}{I(T')} 
\sum_{l_1,\ldots,l_h}
\pfrac{x 2^{-h/2}}{l_1 \cdots l_h}^\beta.
\]
Suppose that $j\ge 2$ and that $l_1,\ldots,l_{j-1}$ have been chosen.
If node $j$ is a leaf of $T$, then $l_j=1$.  
Otherwise, using \eqref{qlq},
the primes $q_k$ for $k\in B_j$ are determined by $l_1,\ldots,l_{j-1}$.
Moreover, a prime $r|l_j$ must equal one of these primes $q_k$.  Hence
\[
 \sum_{l_j} l_j^{-\beta} \le \prod_{k\in B_j} \( 1 - q_k^{-\beta} \)^{-1}
\le \(1-2^{-\beta}\)^{-1} \(1-3^{-\beta}\)^{1-|B_j|}.
\]

By Borchardt's formula\footnote{commonly known as Cayley's formula.}
 \cite{Bor} for counting labelled trees, 
\[
\sum_{T'\in \mathcal{T}} \frac{1}{I(T')} = \frac{n^{n-2}}{(n-1)!} = 
\frac{n^{n-1}}{n!} \le \er^n. 
\]
Since $\sum_{|B_j|\ge 1} (|B_j|-1)=h/2-1$, we conclude that
\be\label{fph}
|\{ p\le x: f(p)=h \}| \le \er^{h/2} \(x 2^{-h/2} \)^\beta 
\(1-2^{-\beta}\)^{-h/2} \(1-3^{-\beta}\)^{-h/2} 
\ee

Taking $\beta=0.37$, the right side of \eqref{fph} is
$\le x^{0.37} (27.8371)^{h/2}\le x^{0.999}$ for $h\le 0.378 \log x$.
For the second part of Theorem \ref{Nx}, assume $h\le (2/5)\log x$ (if 
$h\ge 3\log x$, there are no such $p$ and for $(2/5)\log x < h\le 3\log x$,
$(\frac{6\log x}{h})^h > x$).
Take $\beta=h/\log x$.  
Since $0< \beta \le 2/5$, $2^\b-1\ge \b \log 2$ and $1-3^{-\b}
\ge 0.889 \beta$.  Hence, the right side of \eqref{fph}
is $\le x^{\b} (4.412/\b^2)^{h/2}$.

We remark that numerical improvements are possible by refining the above
analysis; e.g. using the fact that leafs of $T'$ must be labelled with
a Fermat prime.

%
%
%
\section{Lower bounds for $H(p)$: proof of Theorem \ref{Prattlower}}
\label{sec:Prattlower}
%
%
%

We show part (b) first, as the proof is much easier.

\begin{proof}[Proof of Theorem \ref{Prattlower} (b)]
Let $c'$ and $\th'$ satisfy $\th' > 1/3$ and
$c < c' < \frac{1}{-\log \th'} < \frac{1}{-\log \th},$
and define $K$ by $8\theta^K = \th-\th'.$
Let $x_0$ be large, depending on $K,c,c',\th,\th'$ and put
$c''=c'\log_2 (x_0^3)$.  Let
$\mathcal{P}=\{p : H(p)\ge c'\log_2 p - c''\}.$
In particular, $\mathcal{P}$ contains all primes $\le x_0^3$.  We
shall prove
\be\label{Prattlow2}
Q(x) := \left| \mathcal{P} \cap (x/2,x] \right| \ge\frac{x}{(\log
    x)^K}
\ee for $x\ge x_0$, which implies the desired conclusion (since $c'>c$). By the
prime number theorem and the fact that $K>1$, if $x_0$ is large
enough then \eqref{Prattlow2} holds for $x_0 \le x \le x_0^3$.
Suppose $y\ge x_0^3$ and \eqref{Prattlow2} holds for $x_0 \le x \le
y$.  Assume $y<x\le 2y$ and put $I = \mathcal{P} \cap
(x^{\th'},x^{\th}]$.
  Suppose that $x/2<p\le x$, and that $q|p-1$, where
$q\in I$.  Then
\begin{align*}
H(p) &\ge 1 + H(q) \ge 1 + c' \log_2 q - c'' \\
&\ge 1+ c' \log_2 x + c' \log \th' - c'' > c'\log_2 p - c'',
\end{align*}
so that $p\in \mathcal{P}$.
For $x/2<p\le x$, $p-1$ is divisible by at most two primes from
  $I$, hence
\[
Q(x) \ge \frac12 \sum_{q\in I} \Big( \pi(x;q,1)-\pi(x/2;q,1) \Big) 
\ge \frac{x}{4\log x} \sum_{q\in I} \frac{1}{q} + O\pfrac{x}{(\log x)^{K+1}}.
\]
Since \eqref{Prattlow2} holds for $x^{\th'} < y \le x^{\th}$, the sum
on $q\in I$ is
\[
\ge \sum_{2^j \le
    x^{\th-\th'}} \frac{Q(2^j x^{\th'})}{2^j x^{\th'}}
\ge  \( \frac{(\th-\th')\log x}{\log 2} - 1 \) \frac{1}{(\log
  x^{\th})^K}
\ge \frac{\th-\th'}{\th^K(\log x)^{K-1}}
= \frac{8}{(\log x)^{K-1}}.
\]
Therefore, \eqref{Prattlow2} holds.  By induction on dyadic intervals,
\eqref{Prattlow2} holds for all $x\ge x_0$.
\end{proof}

{\bf Remark.}  The same proof gives,
assuming that \eqref{BVQ} holds with $Q=x^{1-\eps(x)}$ and $R(x)=x(\log x)^{-g(x)}$
where $\eps(x)\to 0$ and $g(x)\to\infty$ as $x\to \infty$, that $H(p) \ge
h(p) \log_2 p$ for infinitely many $p$, where $h(p)\to \infty$ as
$p\to \infty$ (the function $h$ depending on the functions $\eps, g$).


\begin{proof}[Proof of Theorem \ref{Prattlower} (a)]
We proceed by induction as in part (b), but instead we iterate by
many levels in the chain at once rather than one level at a time.
Suppose that
$c < h < c' < 1/(\er^{-1} - \log(\th)).$
For some constant $c''$, described below, let
$\PP=\{p : H(p)\ge c'\log_2 p - c''\}.$
We will show, for some $\delta>0$, that
\be\label{PPx}
\PP(x) := | \{ p\le x : p\in \PP \}| \ge \delta \frac{x}{\log x}.
\ee
Consequently, a positive proportion of primes $p$ satisfy $H(p) > h
\log_2 p$, and Theorem \ref{Prattlower} (a) follows from Theorem \ref{tightleft}.

By Stirling's formula, there is
  an integer $k\ge 2$ such that
$1/c' > \frac{1}{k}(k!)^{1/k} - \log (\th)$.
Let $\a,\b$ satisfy
$\er^{-k/c'} < \b < \th^k \exp \( - (k!)^{1/k} \)$
and $\b \exp \( (k!)^{1/k} \) < \a < \th^k.$
Suppose that $\delta$ is sufficiently small, depending only on the
choice of $c', \th, k, \a, \b$.  Let $x_0$ be sufficiently large,
depending on $c', \th, k, \a, \b, \del$, and put $c'' = c' \log_2
(x_0)$. Observe that \eqref{PPx} holds trivially for $2 \le x \le
x_0$, provided $\del$ is small enough.  Throughout this proof,
constants implied by the $O-$ and $\ll-$symbols may depend on $c',
\th, k, \a, \b$, but not on $\del$.

Next, suppose that $Y\ge x_0$ and that inequality \eqref{PPx} holds for $2 \le
x\le Y$.  Let $S$ be a subset of the primes in $\mathcal{P} \cap
[Y^\b, Y^{\th^k}]$.
Let $M(S)$ be the number of primes $p_0\in (Y,2Y]$ so
that there is a prime chain
$p_k \prec p_{k-1} \prec \cdots \prec p_0$
with $p_k \in S$.  For such $p_0$, we have
\begin{align*}
H(p_0) &\ge k + H(p_k) 
\ge k + c'\log_2 p_k - c'' \\
&= c' \log_2 (2Y) - c'' + c' \log \b + k + O\pfrac{1}{\log Y}
\ge c'\log_2 p_0 - c''
\end{align*}
if $x_0$ is large enough.  We will show, for appropriate $S$, that
\be\label{PLN}
M(S) \ge \del \frac{Y}{\log Y},
\ee
which implies
$P(2Y) \ge P(Y) + M(S) \ge 2\del Y/\log (2Y).$
By induction over dyadic intervals, \eqref{PLN} implies
\eqref{PPx}, and hence Theorem \ref{Prattlower} (a).

To prove \eqref{PLN}, we will consider chains satisfying
not only $p_k\in S$, but also
\be\label{PLpj}
p_{j+1} \le p_j^\th \quad (0\le j\le k-1), \quad p_1 \le Y^{\th}.
\ee
With \eqref{PLpj}, we can use \eqref{BVQ} to accurately count such chains.
We have $M(S) \ge M_1(S) - M_2(S),$
where $M_1(S)$ is the number of chains satisfying
$p_k \in S$ and \eqref{PLpj}, and $M_2(S)$ is the number of
pairs of distinct chains satisfying these conditions with the same
$p_0$.  We begin with
\[
M_1(S) = \sum_{p_k\in S} \sum_{p_{k-1}} \cdots \sum_{p_1} \biggl(
\pi(2Y,p_1,1)-\pi(Y,p_1,1) \biggr),
\]
where $p_k\prec \cdots \prec p_0$ and \eqref{PLpj} in the summations.
By  \eqref{BVQ} and induction on $1\le j\le k$,
\be\label{N1j}
M_1(S) =
\frac{Y}{\log Y} \sum_{p_k} \sum_{p_{k-1}} \cdots \sum_{p_j} \frac{\big(\log_2
  Y^{\th^j} - \log_2 p_j \big)^{j-1}}{p_j (j-1)!} + o\pfrac{Y}{\log Y}.
\ee
Here, we used that for each $p_j$, there are $O(1)$ chains
 $p_{k} \prec \cdots \prec p_j$ with $p_k \ge Y^{\b}$.
By \eqref{N1j} with $j=k$,
\begin{align*}
M_1(\mathcal{P} \cap [Y^{\b},Y^{\a}]) &\ge 
\frac{\del Y}{\log Y}
 \int_{Y^\b}^{Y^{\a}} \frac{\( \log_2 Y^{\a}-\log_2 t
  \)^{k-1}}{(k-1)! t\log t}\, dt + o\pfrac{Y}{\log Y} \\
&= \frac{\del Y}{\log Y} \frac{(\log (\a/\b))^k}{k!} + o\pfrac{Y}{\log Y}.
\end{align*}
By hypothesis, $\log(\a/\b) > (k!)^{1/k}$.  The summands
in \eqref{N1j} (with $j=k$) are $\asymp 1/p_k$ for $Y^\b \le p_k \le
Y^\a$.  Hence,
if $\del$ is small enough, there is a set $S\subseteq \mathcal{P} \cap
[Y^\b,Y^\a]$ such that
\be\label{N1lower}
\sum_{p_k\in S} \frac{1}{p_k} \ll \del, \qquad
M_1(S) \ge \( \del + \del^{3/2} \) \frac{Y}{\log Y}.
\ee

We have $M_2 = M_{2,0}+\cdots+M_{2,k-1}$
where $M_{2,j}$ counts pairs of coupled chains
$$
\begin{matrix}
p_k \prec \cdots \prec p_{j+1} \\
\vspace{0.7mm} \\
p_k' \prec \cdots \prec p_{j+1}'
\end{matrix}
\, \Big> \, p_j \prec \cdots \prec p_0
$$
with each of the two chains satisfying \eqref{PLpj}, $p_{j+1} \ne
p_{j+1}'$ and $p_k, p_k'\in S$.  We further write $M_{2,j} =
M'_{2,j} + M''_{2,j}$, where $M_{2,j}'$ counts pairs of such chains
with $p_j \le p_{j+1} p_{j+1}' Y^{\del^2}$.  As before, for each
pair $(p_{j+1}$, $p_{j+1}')$, there are $O(1)$ choices for $p_k,
p_k', \ldots, p_{j+2}, p_{j+2}'$.  For $M_{2,0}'$,
$p_1p_1' \ge Y^{1-\del^2}$, and for each $p_0$, there are $O(1)$ choices
for $p_1,p_1'$.  By sieve methods (e.g. Theorem 2.2 of \cite{HR}),
\begin{align*}
M_{2,0}' &\ll
\sum_{1\le k\le Y^{\del^2}} | \{ n\le Y : n\equiv 1 \!\!\! \pmod{k}, P^-(n
(\tfrac{n-1}{k})) >Y^\b \}| \\
&\ll  \sum_{1\le k\le Y^{\del^2}} \frac{Y}{\phi(k)\log^2 Y} 
\ll \frac{\del^2 Y}{\log Y}.
\end{align*}
Here, $P^-(m)$ is the smallest prime factor of $m$.
For $j\ge 1$, an argument similar to that leading to \eqref{N1j}, followed by the same sieve bound, gives
\[
 M_{2,j}' \ll \frac{Y}{\log Y} \sum_{p_{j+1},p_{j+1}'} \sum_{p_j}
\frac{1}{p_j} \ll \sum_{k\le Y^{\del^2}} \; \sum_{\substack{n\le Y, n\equiv 1\!
\!\!\pmod{k} \\ P^-(n(\frac{n-1}{k})>Y^{\b}}} \; \frac{1}{n}\ll \del^2 \frac{Y} 
{\log Y}.
\]

For chains counted by
$M_{2,j}''$, the Brun-Titchmarsh inequality suffices for the
estimations.  When $j\ge 1$ and $p_j$ is given, as before we have
$$
\sum_{p_{j-1}} \cdots \sum_{p_1} \pi(2Y,p_1,1) \ll \frac{Y}{p_j \log Y}.
$$
By partial summation and \eqref{BT}, given $p_{j+1}$ and $p_{j+1}'$,
$$
\sum_{p_j} \frac{1}{p_j} \ll
\frac{1}{p_{j+1} p_{j+1}' \del^2 \log Y} + \frac{\log
  (1/\del)}{p_{j+1} p_{j+1}'} \ll  \frac{\log
  (1/\del)}{p_{j+1} p_{j+1}'}.
$$
For $j+1 \le r \le k-1$, 
\be\label{pr} \sum_{p_r} \frac1{p_r} \ll
\frac{1}{p_{r+1}}, \qquad \sum_{p_r'} \frac{1}{p_r'} \ll
\frac{1}{p'_{r+1}}. 
\ee 
Finally, by \eqref{N1lower}, we arrive at
\be\label{N''} 
M_{2,j}'' \ll \del^2 \log(1/\del) \frac{Y}{\log Y}.
\ee 
In a similar way, when $j=0$, we have by \eqref{BT} and partial
summation,
$$
\sum_{p_1p_1'\le 2Y^{1-\del^2}} \pi(2Y,p_1 p_1',1) \ll
\frac{\log(1/\del)}{p_2 p_2'}.
$$
A second application of \eqref{pr} then gives \eqref{N''} in this case.

Finally, combining our estimates for $M_{2,j}'$ and $M_{2,j}''$, we obtain
$M_2(S) \ll \del^2 \log(1/\del) Y/\log Y.$
Together with \eqref{N1lower}, if $\del$ is small enough then
\eqref{PLN} holds, and this completes the proof.
\end{proof}

%
%
%
\section{Proof of Theorems \ref{Prattupper} and \ref{EGPSconj}}\label{sec:Prattupper}
%
%
%

The proofs of Theorems \ref{Prattupper} and \ref{EGPSconj} rely on the fact
that the largest prime factor of $p-1$ cannot be too large
too often.   At the core is a sieve upper bound for $k$-tuples of
primes which is uniform in $k$, and careful averages of the associated
singular series.  There is a potentially troublesome factor $2^k k!$ in
the sieve estimate, which is partly overcome by observing that
if $H(p)$ is large, then there must be a prime chain in the Pratt
tree for $p$ which is very condensed in a multiplicative sense.

%
%

\begin{lem}\label{sieve}
There is a positive constant $\delta$ so that the following holds.
Let $a_1,\ldots,a_k$ be positive integers, let $b_1,\ldots,b_k$ be
integers with $(a_j,b_j)=1$ for all $j$, 
and let $\xi(p)$ be the number of solutions of
$\prod_{i=1}^k (a_i n+b_i) \equiv 0\pmod{p}$.
If $x\ge 10$, $1\le k\le \delta \frac{\log x}{\log_2 x}$ and
\[
B := \sum_{p} \frac{k-\xi(p)}{p}\log p \le \delta \log x,
\]
then the number of integers $n\le x$ for which $a_1 n + b_1,
\ldots, a_k n + b_k$ are all prime and $>k$ is
\[
\ll \frac{2^k k!}{(\log x)^k} x \SS \cdot \exp \( O\pfrac{kB+k^2\log_2
  x}{\log x} \), \qquad
\SS=\prod_{p} \( 1 - \frac{\xi(p)}{p}\) \(1 - \frac{1}{p}\)^{-k}.
\]
\end{lem}

\begin{proof}
Since $\xi(p)=k$ for
large $p$, $\SS>0$ if and only if
$\xi(p)<p$ for all $p$.  Also, $\xi(p)\le k$ for all $p$.
Hence, if $\SS=0$, the number of $n$ 
is zero.  If $\SS>0$, 
Montgomery's large sieve estimate \cite[Th\'eor\`eme 6]{Bo} implies
that the number of $n$ in question is $\ll x/G(\sqrt{x})$, where
$$
G(z) = \sum_{n\le z} g(n), \qquad g(n)=\mu^2(n) \prod_{p|n}
\frac{\xi(p)}{p-\xi(p)},
$$
and $\mu$ is the M\"obius function.
For fixed $k$, the argument in \cite[\S 5.3]{HR} gives
$G(z) \sim (\log z)^k/(k! \SS)$.  We sketch how to make explicit the dependence on $k$.
By the argument in \cite[p. 147--148]{HR}, 
\begin{align*}
\sum_{d\le z} g(d)\log d
&= \sum_{d\le z} g(d) \sum_{p\le z/d} \frac{\xi(p)\log p}{p} +
\sum_{h\le z} g(h) \sum_{\substack{p|h \\ p>z/h}} \frac{\xi(p)\log
  p}{p} \\
&= k\sum_{d\le z} g(d) \log \frac{z}{d} + O\(G(z)(B+k\log_2 z) \).
\end{align*}
Adding the sum on the right side to both sides yields
\[
G(z) \log z = (k+1) \int_1^z \frac{G(t)}{t}\, dt +  r(z) G(z)\log z,
\] 
where $r(z) \ll \frac{B+k\log_2 z}{\log z}$.  If
$\delta$ is small enough and $z\ge \sqrt{x}$, then $|r(z)| \le
\frac12$.  By the argument in \cite[p. 150]{HR},  for
some constant $D$ and for $z\ge \sqrt{x}$,
\[
(1-r(z))\frac{G(z)}{\log^k z} 
= D \exp \left\{ O\pfrac{kB+k^2\log_2 z}{\log z} \right\}.
\]
By the argument on \cite[p. 151--152]{HR}, $D^{-1} = k! \SS$.
Taking $z=\sqrt{x}$ completes the proof.
\end{proof}

%
%

For given integers $m_1,\ldots,m_{k-1} \ge 2$, we will apply
Lemma \ref{sieve} with the forms 
$f_1(n)=n$, $f_{j+1}(n)=m_{j} f_j(n)+1 \ (1\le j\le k-1).$
We have $f_j(n)=a_j n + b_j$, where
\be\label{ajbj}
 a_j=m_1\cdots m_{j-1} \quad (j\ge 1), \qquad b_1=0,
\quad  b_j=1 + \sum_{i=2}^{j-1} m_i\cdots m_{j-1} \;\; (j\ge 2). 
\ee
Clearly, $(a_j,b_j)=1$.
Let $\SS(\mm)=\SS$ be the associated singular series and let
$\xi(p,\mm)=\xi(p)$.


\begin{lem}\label{singupper}
There is a positive constant $c_1$ so that
$\SS(\mm) \ll (c_1 \log_2 (4m_1\cdots m_{k-1}))^{k-1}$.
Also,
\[
\sum_p \frac{k-\xi(p,\mm)}{p} \log p \le k\( \log_2 (4m_1 \cdots
m_{k-1})+O(1) \).
\]
\end{lem}

\begin{proof}  
We have $\xi(p,\mm)=k$ if $p\nmid N$, where
$N = m_1 \cdots m_{k-1} \prod_{i<j} |a_i b_j - a_j b_i|.$
Also, $\xi(p,\mm)\ge 1$ for all $p$.
Let $x=m_1 \cdots  m_{k-1} \ge 2^{k-1}$.  
By \eqref{ajbj}, $a_j\le x$ and
$b_j \le 1 + \sum_{j=1}^{k-2} x/2^j \le x$ for each $j$.  Thus,
$N\le x^{k(k-1)+1} \le \exp \{ O(\log^3 x) \}$.  Since
$1-k/p \le (1-1/p)^k$ for $p>k$, if $\SS(\mm)>0$ then 
$\SS(\mm)\le \prod_{p|N}(1-1/p)^{1-k}\le (c_1\log_2 N)^{k-1}$ for a constant
$c_1$.  The second bound follows from $\sum_{p|N} (\log p)/p \le \log_2 N+O(1)$.
\end{proof}


\begin{lem}\label{rhopm}  Let $k\ge 1$ and $\II \subseteq
  \{1,\ldots,k-1\}$.  If the variables $m_i$ are fixed $(1\le i\le
  k-1, i\not\in \II)$, then for any prime $p$,
$$
\sum_{0\le m_i <p \; (i\in \II)} \xi(p,(m_1,\ldots,m_{k-1})) \ge
p^{|\II|+1} - (p-1)^{|\II|+1}.
$$
\end{lem}

\begin{proof}   Fix $p$ and let $N(k,\II)$ be the sum on the left side. 
We use induction on $k$, the case $k=1$ being trivial.  Suppose $k\ge 2$
and the lemma holds with $k$ replaced by $k-1$.  If $k-1 \not\in \II$, then
$\xi(p,(m_1,\ldots,m_{k-1})) \ge \xi(p,(m_1,\ldots,m_{k-2}))$ implies
$N(k,\II) \ge N(k-1,\II)$.  If $k-1 \in \II$, then $N(k,\II)$ counts the number of 
$(|\II|+1)-$tuples $(m_i (i\in \II), n)$ modulo $p$ with
$p|f_1(n)\cdots f_{k-1}(n)(m_{k-1}f_{k-1}(n)+1)$.  The number of tuples with $p|f_1(n)\cdots f_{k-1}(n)$ is $p N(k-1,\II-\{k-1\})$ and the number of remaining tuples is $p^{|\II|}-N(k-1,\II-\{k-1\})$.  By the inductive hypothesis,
$N(k,I) = p^{|\II|} + (p-1) N(k-1,\II-\{k-1\}) \ge p^{|\II|+1} - (p-1)^{|\II|+1}.$
\end{proof}

%
%

\begin{lem}\label{singular}  Let $k\ge 4$, and suppose that $M_i,N_i$ are
  integers satisfying $M_i\ge 2$ and
$2\le N_i \le 2kM_i$ for $1\le i\le k-1$.  For some positive constant $c_2$, we have
\[
\sum_{\substack{N_i < m_i \le N_i+M_i \\ (1\le i\le k-1)}}
\SS(\mm) \ll M_1 \cdots M_{k-1} \( c_2 \log k \)^{b}
\exp \left\{ O\pfrac{k \log_2 k}{\log k} \right\},
\]
where $b$ is the number of variables $M_i$ which are $\le 2^{k^2\log^3 k}$.
\end{lem}

\begin{proof}
Let $L=\fl{\log k}+1$ and $r=k^2L$. We will perform a precise
averaging of the factors in $\SS(\mm)$ for primes $p\le r$, and
use crude estimates for larger $p$.
If $p \nmid m_1\cdots m_{k-1}$, each congruence $f_j(n)\equiv
0\pmod{p}$ has exactly one solution.
For $h>j$, $f_j(n)\equiv 0\pmod{p}$ and $f_h(n)\equiv 0\pmod{p}$ have
a common solution if and only if $p|(a_j b_h-a_h b_j)$.
Write
\be\label{gjh}
a_j b_h - a_h b_j = m_1 \cdots m_{j-1} g_{j,h}(\mm), \qquad
g_{j,h}(\mm) := 1 + \sum_{i=j+1}^{h-1} m_i \cdots m_{h-1}.
\ee
Define
$$
\psi_r(n) = \prod_{\substack{p|n \\ p>r}} \frac{p}{p-1}, \qquad
G_{j,h}(\mm) = \!\!\!\!\!\prod_{\substack{p>r, p|g_{j,h}(\mm) \\ p\nmid
    g_{i,h}(\mm) \; (j+1\le i\le h-2) \\ p\nmid m_1\cdots m_{k-1}}}
\!\!\!\!\!\!\!\! p.
$$
We then have
\[
\prod_{\substack{p\nmid m_1\cdots m_{k-1}\\p>r}} \(1-\frac{\xi(p,\mm)}{p}\)
\(1-\frac1{p}\)^{-k} \le \prod_{\substack{1\le j<h\le k-1
    \\ h\ge j+2}} \psi_r(G_{j,h}(\mm)).
\]
Let 
\[
\mathcal{J}=\{(j,h):1\le j<h\le k-1, h\ge j+2, \max_{j+1\le i\le h-1} M_i 
> 2^{k^2\log^3 k} \},
\]
and put $J=|\mathcal{J}| \le \frac{(k-3)(k-2)}{2}.$
Also let $\II=\{ i: M_i > 2^{k^2\log^3 k} \}$.  Write
$\II=\bigcup_{a=1}^A ([i_a,i'_a]\cap \NN)$, where $i_{a+1}\ge i'_a+2$ for each $a$.
For each $a$ and $2+i'_a \le h\le i_{a+1}$ (so $h\not\in \II$),
\[
\prod_{i_a' \le j\le h-2} \psi_r(G_{j,h}(\mm)) =\psi_r(G_h),
\quad G_h=\prod_{i_a' \le j\le h-2} G_{j,h}(\mm). 
\]
Since $G_h\le (k 2^{k^2\log^3 k})^k$, $G_h$ has $O(k^3\log^3 k)$ prime factors,
and thus for some constant $C>1$,
$\psi_r(G_h)\le \exp \{ \sum_{p|G_h,p>r} \frac{1}{p-1} \} \le C$.
There are $b$ such numbers $h$.  Hence, by H\"older's inequality,
\be\label{Holder}
\begin{split}
\sum_{\mm} \SS(\mm) &\le C^b \sum_{\mm} \prod_{p\le r}
  \(1-\frac{\xi(p,\mm)}{p} \) \(1-\frac{1}{p}\)^{-k} \;\;
  \prod_{i=1}^{k-1} \psi_r(m_i)^{k-1} \prod_{(j,h)\in\mathcal{J}}
  \psi_r(G_{j,h}(\mm)) \\
&\le C^b \biggl( \sum_{\mm} \biggl[ \prod_{p\le r}  \(1-\frac{\xi(p,\mm)}{p}
  \) \(1-\frac{1}{p}\)^{-k} \biggr]^{\frac{L}{L-1}}
  \biggr)^{1-\frac{1}{L}}\\
&\qquad \times  \prod_{i=1}^{k-1} \( \sum_{\mm}
  \psi_r(m_i)^{2L(k-1)^2} \)^{\frac{1}{2L(k-1)}} \prod_{(j,h)\in\mathcal{J}}
 \( \sum_{\mm} \psi_r(G_{j,h}(\mm))^{2JL} \)^{\frac{1}{2JL}}.
\end{split}
\ee
If we write $\psi_r^s = 1 \ast \beta_s$,
then $\beta_s$ is multiplicative and supported on square-free
integers composed of primes $>r$.  Furthermore, if $p>r\ge s+1$, then
\be\label{betas}
\beta_s(p) = \pfrac{p}{p-1}^s - 1 \le \er^{s/(p-1)}-1 \le \frac{4s}{p}.
\ee
Thus, for each $i$,
\be\label{Holder2}
\begin{split}
\sum_{N_i < m_i\le N_i+M_i} \psi_r(m_i)^{2L(k-1)^2} &\le \sum_{d\le
  N_i + M_i}
  \beta_{2L(k-1)^2}(d) \frac{N_i + M_i}{d} \\
&\le (2k+1)M_i \prod_{p>r} \(1 + \frac{\beta_{2L(k-1)^2}(p)}{p} \) 
\ll kM_i.
\end{split}
\ee
For fixed $(j,h)\in\mathcal{J}$, let $M_l=\max
(M_{j+1},\ldots,M_{h-1}) > 2^{k^2\log^3 k}$ and write
\be\label{gjhl}
g_{j,h}(\mm) = m_l (m_{l+1}\cdots m_{h-1}) g_{j,l}(\mm) + g_{l,h}(\mm).
\ee
We'll use
$$
G_{j,h}(\mm) | G_{j,h}'(\mm) := \prod_{\substack{p>r, \,
p|g_{j,h}(\mm) \\ p\nmid m_1\cdots m_{k-1} g_{l,h}(\mm)}} p,
$$
and note that $G_{j,h}'(\mm) \le g_{j,h}(\mm) \le k (2k+1)^k M_{j+1}\ldots
M_{h-1} \le (6kM_l)^{k}$ by \eqref{gjh}. 
Fix all of $m_{j+1},\ldots,m_{h-1}$ except for $m_l$.
By \eqref{betas} and \eqref{gjhl},
\be\label{Holder3}
\begin{split}
\sum_{m_l} \psi_r(G_{j,h}'(\mm))^{2JL} &= \!\!\!\! \sum_{\substack{d\le
   (6kM_l)^{k} \\ (d,g_{l,h}(\mm) \prod_{i\ne l} m_i)=1}}
   \!\!\!\!\! \beta_{2JL}(d) \sum_{\substack{N_l<m_l\le N_l+M_l \\
d|G_{j,h}'(\mm)}} 1 
\le \sum_{d\le (6kM_l)^{k}} \( \frac{M_l}{d} + 1 \) \beta_{2JL}(d) \\
&\le M_l \prod_{p>r} \( 1 + \frac{8JL}{p^2} \) +
  \prod_{r<p\le (6kM_l)^{k}} \( 1 + \frac{8JL}{p} \) \ll M_l.
\end{split}\ee
Also,
\[
\biggl[ \prod_{p\le r}
  \(1-\frac{\xi(p,\mm)}{p} \) \(1-\frac{1}{p}\)^{-k}
  \biggr]^{\frac{L}{L-1}-1} \le \prod_{p\le r}
  \(1-\frac1{p}\)^{-\frac{k}{L-1}} 
=\exp\left\{ O\pfrac{k\log_2 k}{\log k} \right\}.
\]
Therefore, by \eqref{Holder}, \eqref{Holder2} and \eqref{Holder3},
\[
\sum_{\mm} \SS(\mm) \ll C^b \(M_1\cdots M_{k-1}\)^{\frac{1}{L}}
S^{1-\frac{1}{L}} \exp\left\{ O\pfrac{k\log_2 k}{\log k} \right\},
\]
where
\[
S = \sum_{\mm} \prod_{p\le r}
  \(1-\frac{\xi(p,\mm)}{p} \) \(1-\frac{1}{p}\)^{-k}.
\]
Let $M'=\prod_{p\le r} p$.
Since $M' \le \er^{2r}$,  for each $i\in\II$,
$M_i \gg k M'$.  Hence,
 the number of $m_i\in (N_i,N_i+M_i]$ lying in a given residue class modulo $M'$
is $\le M_i/M'+1 \le (1+O(1/k))M_i/M'$.
Thus, by Lemma \ref{rhopm} and the Chinese Remainder Theorem,
\begin{align*}
S &\le \sum_{\substack{N_i < m_i \le N_i+M_i \\ (i\not\in \II)}} \prod_{i\in \II}
  \frac{M_i}{M'} \(1+O\pfrac{1}{k} \) \sum_{\substack{m_i\bmod M' \\
      (i\in \II)}}
  \prod_{p\le r} \(1-\frac{\xi(p,\mm)}{p} \) \(1-\frac{1}{p}\)^{-k} \\
&\ll \prod_{i\in \II} M_i \sum_{\substack{N_i < m_i \le N_i+M_i \\
    (i\not\in \II)}}
  \prod_{p\le r} \frac{1}{p^{|\II|}} \pfrac{p}{p-1}^k \biggl[
  p^{|\II|} - \frac{1}{p} \sum_{\substack{0\le m_i<p\\ (i\in \II)}} \xi(p,\mm)
  \biggr] \\
&\ll M_1 \cdots M_{k-1} \prod_{p\le r} \pfrac{p}{p-1}^{k-1-|\II|}.
\end{align*}
Noting that $b=k-1-|\II|$, the lemma follows from Mertens' estimate.
\end{proof}

%
%

\begin{thm}\label{condensed_chain}
Suppose that $\eta>0$, $r\ge 1$, and $l$ and $x$ are
sufficiently large as a function of $\eta$.  There are 
\[
\ll \frac{x}{\log x} (2\eta \er^{1+\eta})^{l/2}+\sum_{j=1}^r x(\log_2 x)^{O(jl)}
\pfrac{(jl)^{3+\eta}}{\log x}^{\lfloor jl/(\log_2 jl)^2 \rfloor}
\]
primes $p\le x$, such that there is a prime chain $p_{rl} \prec p_{rl-1} \prec
\cdots \prec p_0=p$ with $p_{rl}>x^{(r+1)^{-\eta}}$. 
\end{thm}

\begin{proof} Suppose $2\eta \er^{1+\eta}<1$ and $rl\le (\log x)^{1/(3+\eta)}$,
else the theorem is trivial.  Put
$k_j = jl$  and $x_j = x^{(j+1)^{-\eta}}$ for $0\le j\le r$.
Suppose $p\le x$ and there are even integers $h_1,\ldots,h_{k_r}$ so that
\be\label{ph}
p=p_0=h_1p_1+1, p_1=h_2p_2+1, \ldots, p_{k_r-1}=h_{k_r} p_{k_r}+1,
\ee
with $p_0,\ldots,p_{k_r}$ prime and $p_{k_r} \ge x_r$.  
The vector $(h_1,\ldots,h_{k_r})$ may not be unique, but we associate
to each such $p$ a single such vector.  Each $p$ lies in 
$Q_1\cup \cdots \cup Q_r$, where
$Q_j$ is the set of primes $p$ so that
$p_{k_i} < x_i \; (i<j)$ and $p_{k_j} \ge x_j$.
By assumption, $k_r \le (\log x_r)^{1/3}$.

Fix $j$ and even integers $h_1,\ldots,h_{k_j}$ satisfying
$h_1 \cdots h_{k_j} \le x/x_j$.  By Lemma
\ref{singupper},
$$
\sum_{p} \frac{k_j-\xi(p;(h_{k_j},\ldots,h_1))}{p}\log p \ll k_j
\log_2 x \ll (\log x_r)^{1/3} \log_2 x.
$$
By Lemma \ref{sieve} and
Stirling's formula, the number of $p=p_0\le x$ satisfying \eqref{ph}
is 
\be\label{upperl} \ll \frac{x}{h_1 \cdots h_{k_j}} \frac{ (2
k_j/\er)^{k_j+3/2} \SS(h_{k_j},\ldots,h_1)}{\(\log x_j \)^{k_j+1}}.
\ee 
Let $1\le b_j\le
k_j$ be a parameter to be chosen later, and put $A_j=2^{2k_j^2\log^3 k_j}$.
Let $Q_{j,1}$ be the set of
$p\in Q_j$ for which at least $b_j$ of the variables
$h_1,\ldots,h_{k_j}$ are $\le A_j$, and $Q_{j,2}=Q_{j}
\backslash Q_{j,1}$.

To estimate $|Q_{j,1}|$, fix a set $\BB \subseteq \{1,\ldots,k_j\}$
of size $b_j$ so that $h_i\le A_j$ for each $i\in \BB$.  Let $\II=\{1\le
i\le k_j: i\not\in \BB\}$ and, for $0\le i\le j-1$, put
$a_i = | \BB \cap \{k_i+1,\ldots,k_{i+1} \}|$, 
$\II_i = \II \cap \{k_i+1,\ldots,k_{i+1} \}.$
By the definition of $Q_j$,
\be\label{hg} \prod_{g\in \II_i \cup \cdots \cup
\II_{j-1}} h_g \le h_{k_{i}+1} \cdots h_{k_j}  \le \frac{x_i}{x_j}
\quad (0\le i\le j-1). 
\ee 
Since $h_i\ge 2$ for all $i$,
\be\label{smallhg} \prod_{g\in \BB} \sum_{2\le h_g \le A_j}
\frac{1}{h_g} \le (2k_j^2\log^3 k_j)^{b_j}. \ee 
Let $\a = {l}/{\log x_j}$.
By \eqref{hg} and the elementary estimate $\sum_{2\le h\le y}
h^{-1-s} \le 1/s$,
\begin{align*}
\sum_{h_g \; (g\in \II)} \frac{1}{\prod_{g\in \II} h_g}
&\le \sum_{h_g \ge 2 \; (g\in \II)} \frac{1}{\prod_{g\in \II} h_g}
 \prod_{i=0}^{j-1} \pfrac{x_i}{x_j}^{\a} \frac{1}{\prod_{g\in \II_i\cup
    \cdots \cup \II_{j-1}} h_g^{\a}} \\
&=\prod_{i=0}^{j-1} \pfrac{x_i}{x_j}^{\a} \sum_{h_g \ge 2 \; (g\in
  \II_i)} \frac{1}{\prod_{g\in \II_i} h_g^{1+(i+1)\a}} \\
&\le \prod_{i=0}^{j-1} \pfrac{x_i}{x_j}^{\a}
  \pfrac{1}{(i+1)\a}^{k_{i+1}-k_i-a_i}\\
&= \pfrac{1}{\a}^{k_j-b_j} \frac{1^{a_0}2^{a_1}\cdots
    j^{a_{j-1}}}{(j!)^l} \exp\left\{
 l \sum_{i=0}^{j} \left[ \pfrac{j+1}{i+1}^\eta-1 \right] \right\}.
\end{align*}
The last sum is $\le \frac{\eta}{1-\eta} (j+1) \le  2j$.
Also, $1^{a_0}2^{a_1}\cdots j^{a_{j-1}} \le j^{b_j}$ and $j! \ge
(j/\er)^j$.  Hence,
\be\label{largehg} \sum_{h_g \; (g\in \II)}
\frac{1}{\prod_{g\in \II} h_g} \le \er^{3k_j} \pfrac{\log
x_j}{k_j}^{k_j-b_j}.
\ee
The number of choices for $\BB$ is $\binom{k_j}{b_j} \le 
(\er k_j/b_j)^{b_j}$.
By \eqref{upperl}, \eqref{smallhg}, \eqref{largehg}, and Lemma
\ref{singupper}, 
\be\label{Qj1}
\begin{split}
|Q_{j,1}| &\ll \frac{x \(c_1 k_j \log_2 x \)^{k_j+3/2}}
{(\log x_j)^{k_j+1}} \pfrac{2\er k_j^3 \log^3 k_j}{b_j}^{b_j}
\er^{3k_j}  \pfrac{\log x_j}{k_j}^{k_j-b_j} \\
&= \frac{x}{\log x_j} k_j^{3/2} \( c_1 \er^3 \log_2 x \)^{k_j+3/2}
  \pfrac{2\er k_j^3 (k_j/b_j) \log^3 k_j (j+1)^\eta}{\log x}^{b_j} \\
&\ll x \exp \biggl\{ O(k_j\log_3 x ) + b_j \biggl[
(3 + \eta) \log k_j + \log\pfrac{k_j}{b_j} - \log_2 x
 \biggr]  \biggr\}.
\end{split}
\ee

We next estimate $|Q_{j,2}|$.   Place each variable $h_i$ into an
interval $J_i$.  If $h_i \le A_j$, then take $J_i=(2^{l_i-1},2^{l_i}]$ for
an integer $l_i\ge 1$, and if $h_i > A_j$, then take
$$
J_i=\left (\lfloor A_j(1+1/k_j)^{l_i-1} \rfloor,
  \lfloor A_j(1+1/k_j)^{l_i}\rfloor \right]
$$
for some integer $l_i\ge 1$.  For brevity, write $J_i=(H_i,K_i]$ for each $i$.
Since $K_i-H_i \ge H_i/(2k_j)$, there are at most $b_j$ values of
$i$ with $K_i-H_i \le A_j$.  Lemma \ref{singular} then gives
\begin{align*}
\sum_{h_1\in J_1} \cdots \sum_{h_{k_j} \in J_{k_j}}
  &\frac{\SS(h_{k_j},\ldots,k_1)} {h_1 \cdots h_{k_j}} \le
\frac{1}{H_1 \cdots H_{k_j}}
  \sum_{h_1\in J_1} \cdots \sum_{h_{k_j} \in J_{k_j}} \SS(h_{k_j},\ldots,h_1) \\
&\ll (c_2\log k_j)^{b_j}\exp\left\{O\pfrac{k_j\log_2
    k_j}{\log k_j} \right\} \prod_{i=1}^{k_j} \frac{K_i-H_i}{H_i}.
\end{align*}
By our definition of the intervals $J_i$,
\begin{align*}
\prod_{i=1}^{k_j} \frac{K_i-H_i}{H_i} &\le \prod_{\substack{1\le i\le
  k_j \\ K_i-H_i < A_j}}
2\sum_{H_i<h_i\le K_i} \frac{1}{h_i} \prod_{\substack{1\le i\le k_j \\
K_i-H_i \ge A_j}}
\(1+O\pfrac{1}{k_j}\) \sum_{H_i<h_i\le K_i} \frac{1}{h_i} \\
&\ll 2^{b_j} \sum_{h_1\in J_1} \cdots \sum_{h_{k_j} \in J_{k_j}}
\frac{1}{h_1\cdots h_{k_j}}.
\end{align*}
Thus, after summing over all possibilities for $J_1,\ldots,J_{k_j}$,
we obtain by \eqref{upperl}
$$
|Q_{j,2}| \ll \frac{x (2k_j/\er)^{k_j+3/2}}{(\log x_j)^{k_j+1}}
\exp\left\{ O\( \frac{k_j\log_2 k_j}{\log k_j} + b_j \log_2 k_j \) \right\}
 \sum_{h'_1,\ldots h'_{k_j}} \frac{1}{h'_1 \cdots h'_{k_j}},
$$
where $h'_{k_i+1} \cdots h'_{k_{j}} \le 2^{k_{j}-k_i} h_{k_{i}+1}\cdots
h_{k_{j}}\le  2^{k_j} x_i/x_j$ for $0\le i\le j-1$.
For positive $\a_0, \ldots, \a_{j-1}$,
\begin{align*}
\sum_{h'_1,\ldots,h'_{k_j}} \frac{1}{h'_1 \cdots h'_{k_j}} &\le
  \prod_{i=0}^{j-1} \biggl[ \( 2^{k_{j}} \frac{x_i}{x_j} \)^{\a_i}
  \sum_{h'_{k_i+1}, \ldots,h'_{k_{i+1}}=2}^\infty
  \frac{1}{(h_{k_i+1}' \cdots h_{k_{i+1}}')^{1+\a_0+\cdots+\a_i}} \biggr] \\
&\le \prod_{i=0}^{j-1} \( 2^{k_{j}} \frac{x_i}{x_j} \)^{\a_i}
  \pfrac{1}{\a_0+\cdots+\a_i}^{l}.
\end{align*}
If we ignore the factors $2^{k_j\a_i}$, the optimal choice of parameters is
$$
\a_i = \frac{l}{(j+1)^\eta \log x_j} \biggl[ \frac{(i+2)^\eta
    (i+1)^\eta}{(i+2)^\eta - (i+1)^\eta} - \frac{(i+1)^\eta
    i^\eta}{(i+1)^\eta - i^\eta} \biggr],\qquad i=0,\ldots,j-1.
$$
Since $(i+2)^\eta - (i+1)^\eta \ge \eta (i+2)^{\eta-1}$,
\[
\a_0+ \cdots + \a_i 
\in \Big[ \frac{l(i+1)(i+2)^\eta}{\eta (j+1)^\eta \log x_j},\frac{l(i+2)(i+1)^\eta}{\eta (j+1)^\eta \log x_j}\Big].
\]
Recalling $k_j=jl$, the sum on $h_1',\ldots,h_{k_j}'$ is at most
\[
2^{k_j(\a_0+\cdots + \a_{j-1})} \exp \left\{ \sum_{i=0}^{j-1} \a_i
\biggl[ \pfrac{j+1}{i+1}^\eta - 1 \biggr]\log x_j \right\} 
\pfrac{\eta
  (j+1)^\eta \log x_j}{l}^{k_j} \frac{1}{(j!)^l ((j+1)!)^{\eta l}}.
\]
The exponential
factor is $\er^{lj}=\er^{k_j}$ and $(j+1)!\ge e^{-j-1}(j+1)^{j+1}$, so
$$
\sum_{h'_1,\ldots,h'_{k_j}} \frac{1}{h'_1 \cdots h'_{k_j}} \le
2^{2k_j^2/(\eta \log x_j)} \pfrac{\eta \er^{2+\eta} \log x_j}{k_j}^{k_j}.
$$
Therefore,
\be\label{Qj2}
|Q_{j,2}| \ll \frac{x}{\log x} \( 2 \eta \er^{1+\eta} \)^{k_j}
\exp\left\{ O\( \frac{k_j\log_2 k_j}{\log k_j} + b_j \log_2 k_j \) \right\}.
\ee
Finally, put $b_j=\fl{k_j/(\log_2 k_j)^2}$, and sum the inequalities
\eqref{Qj1} and \eqref{Qj2} for $1\le j\le r$.
\end{proof}

\begin{proof}[Proof of Theorem \ref{Prattupper}]
Let $\eta=0.15718$,
$l=\fl{(\log x)^{\eps}}$ and $r = \fl{(\log x)^{\beta}},$
where $\eps$ and $\b$ are fixed and satisfy $0<\eps+\b<\frac{1}{3+\eta}.$
Then $\log x_r \asymp (\log x)^{1-\eta \beta}.$  
For the primes $p$ not counted in Theorem \ref{condensed_chain},
the primes at level $rl$ of the Pratt tree are all $<x_r$, so
$H(p) \le \frac{\log x_r}{\log 2}+1 + rl \ll (\log x)^{0.95022}$
if we take $\b$ sufficiently close to $\frac{1}{3+\eta}$.
By Theorem \ref{condensed_chain}, the number of exceptional
primes $p\le x$ is $O(x\exp\{ - (\log x)^{\delta} \})$ for some $\delta>0$.
\end{proof}

\begin{proof}[Proof of Theorem \ref{EGPSconj}]
Let $x$ be large, $x/\log x<n\le x$ and suppose there is a prime $p>x^{\eps/2}$
such that $p|\phi_k(n)$.  Then either (i) there is a prime $q>x^{\eps/2}$
and $0\le j\le k$ such that $q^2|\phi_j(n)$, or (ii) there is a
prime chain $p=p_k \prec p_{k-1} \prec \cdots \prec p_1 \prec p_0$ with $p_0|n$.
In case (i), let $j$ be the smallest such index.
Using the uniform estimate 
\[
\sum_{\substack{p\le x \\ p\equiv 1\!\!\!\pmod{m}}} \frac{1}{p} \ll \frac{\log_2 x}
{\phi(m)},
\]
coming from the Brun-Titchmarsh inequality, the number of integers
in category (i) is
\begin{align*}
&\le \sum_{q>x^{\eps/2}} \frac{x}{q^2} + 
  \sum_{j=1}^k \sum_{x^{\eps/2} < q\le x} \; 
\sum_{\substack{p_{j-1}\equiv 1\!\!\!\!
  \pmod{q^2} \\ p_{j-1}\le x}} \;
  \sum_{\substack{p_{j-2}\equiv 1\!\!\!\!\pmod{p_{j-1}} \\ p_{j-2}\le x}} 
  \cdots \sum_{\substack{p_0\equiv 1\!\!\!\!\pmod{p_1} \\ p_0\le x}}
  \frac{x}{p_0} \\
&\ll_{\eps,k} \frac{x^{1-\eps/2}}{\log x} + \sum_{j=1}^k 
  \frac{x^{1-\eps/2}(\log_2 x)^j}{\log x} \ll_{\eps,k} x^{1-\eps/2}.
\end{align*}
Consider $n$ in category (ii).  Take
$\eta=\frac17$, let $r$ be the smallest integer with $(r+1)^{-\eta}<\eps/2$,
let $l$ be sufficiently large, $l\le \log_2 x$ and $k=rl$.  By Theorem \ref{condensed_chain},
for $x^{\eps/2}<y\le x$, the number of $p_0\le y$ is
$O(y/\log^2 y + y (2\eta \er^{1+\eta})^{-l/2}/\log y)$.  By partial summation,
the number of $n$ is $\ll_\eps x (2\eta \er^{1+\eta})^{-l/2}$.  Taking
$l$ large enough, depending on $\eps$ and $\delta$, completes the proof.
\end{proof}

%
%
%
%
\section{Stochastic model of Pratt trees}\label{sec:model}
%
%
%
%

In this section, we develop a model of the Pratt trees which
explains Conjectures \ref{Dpnormal1} and \ref{Dpnormal2}.  Factor $n$
as $n=\prod_{j=1}^{\Omega(n)} p_j(n)$, with $p_1(n) \ge   p_2(n) \ge
\cdots$.  Put $p_j(n)=1$ for $j>\Omega(n)$ and let
$$
S(n) = \( \frac{\log p_1(n)}{\log n}, \frac{\log p_2(n)}{\log n}, \ldots\).
$$
The distribution of the first component of $S(n)$ has been greatly
studied, the results having wide application in the theory of numbers
(see e.g. the comprehensive survey article \cite{HT}).
We have\footnote{If $\mathcal{B} \subseteq \mathcal{A}\subseteq \NN$, we say
 $\PPP(n\in \mathcal{B} | n\in \mathcal{A})=\a$ if
$\displaystyle \lim_{x\to\infty} \frac{| \{ n\in \mathcal{B} : n\le x \}|}{ |\{n\in 
\mathcal{A}: n\le x \}|} = \a.$}
 $\PPP ( \log p_1(n) \le \frac{1}{u}\log n ) = \rho(u),$
where $\rho$
is the \emph{Dickman function}, the unique continuous solution of the
differential-delay equations
$\rho(u)=1 \ (0\le u\le 1)$, $u\rho'(u)=-\rho(u-1) \ (u>1)$.
The complete distribution of $S(n)$, found by
Billingsly in 1972 \cite{Bi}, corresponds to the Poisson-Dirichlet
distribution with parameter 1, $PD(1)$ for short (more precisely,
for each $j$, the first $j$ components of $S(n)$ are distributed as
the first $j$ components of the $PD(1)$ distribution). 
 The joint distribution of the
components in the $PD(1)$ distribution can easily be expressed in terms of
$\rho$.
There is a simpler characterization of
the distribution, found by Donnelly and Grimmett \cite{DG}.  Let
$U_1, U_2, \ldots$ be independent random variables with uniform
distribution on $[0,1]$. Let $\xx=(x_1,x_2,\ldots)$ be the infinite
dimensional vector formed from the decreasing rearrangement of the
numbers 
\be\label{Un} y_1=U_1, \quad y_2=(1-U_1)U_2, \quad
y_3=(1-U_1)(1-U_2)U_3, \ldots. 
\ee 
Then $\xx$ has the $PD(1)$
distribution. The paper \cite{DG} gives a simple, transparent proof
that $(x_1,\ldots,x_k)$ and the first $k$ components of $S(n)$ have
the same distribution. 

Since $\sum x_i = 1$ with probability 1,
we can interpret the $PD(1)$ distribution as a
random partition of the unit interval $[0,1]$ into an infinite number
of parts achieved by cutting $[0,1]$ at a random place (with uniform
distribution), then
cutting the right sub-interval at a random place, and so on.
\begin{conj}\label{PD(p-1)}
As $p$ runs over the set of primes, $S(p-1)$ has $PD(1)$ distribution.
\end{conj}

Conjecture \ref{PD(p-1)} is widely believed, and is a simple consequence of
EH.   Unconditionally, we know little about
primes in progressions to very large moduli. 
Assume that $S(p-1)$ has $PD(1)$ distribution, 
$S(q-1)$ has $PD(1)$ distribution for each prime $q|(p-1)$, the vectors $S(q-1)$ for $q|(p-1)$
are independent, and so forth.  The primes o4n the
first level of the tree, on a logarithmic scale,
correspond to a random partition of $[0,1]$.  The primes on the second
level correspond to randomly partitioning  each of the parts of the
original partition, etc.  The entire procedure corresponds to
what is known as a discrete-time \emph{random fragmentation process}.  Random fragmentation processes have been used to model a variety of physical phenomena
(e.g., genetic mutations, planet formation) and
the growth of certain data structures in computer science.
Discrete time fragmentation processes may be
recast in the language of \emph{branching random walks}, which we now
describe.

Let $M_n$ be the size of the largest object at time $n$. 
Then $M_n$ is a model of $Q_n :=\frac{\log q_n}{\log p}$, where $q_n$
 is the largest prime at level $n$ of the tree.
The event
$\{ M_n < \frac{\log 2}{\log p} \}$ is a model of the statement
``all the primes at level $n$ of the Pratt tree for $p$ are $< 2$'';
that is, $H(p)<n$.  Thus, $H(p)$ is modeled by the random variable
$T(\frac{\log 2}{\log p})$, where 
$T(\eps) = \min \{n : M_n \le \eps \}. $

Assuming EH,
Lamzouri \cite{Lam} showed that $Q_n$ has the same distribution
as $M_n$ for each \emph{fixed} $n$ (he studies the distribution
of $P^+(\phi_n(m))$ for all integers $m$; the same
proofs give the distribution of $Q_n$).  Further,
on  EH, Lamzouri shows that
$\PPP \{ Q_n \le \frac{1}{u} \} = \PPP \{ M_n \le \frac{1}{u} \} = \rho_n(u),$
where, for each fixed $n$,
\be\label{Lamtail}
\rho_n(u) = \pfrac{1+o(1)}{\log_{n-1} (u) \log_n (u)}^u \qquad (u\to \infty),
\ee
with $\log_0 (u) = u$.
Our goal is to understand the distribution of $M_n$ as
$n\to\infty$.  

Create a tree structure from the random fragmentation
process as follows: label the root
node with zero, beneath the root node put an infinite number of
child nodes, each corresponding to one of the fragments of the
initial segment $[0,1]$.  Each of these nodes has an infinite number
of child nodes, corresponding to the fragments in the second step of
the process, and so on.  Each node is labeled with the number $-\log
x$, where $x$ is the fragment size.  This randomly labeled tree
corresponds to a \emph{branching random walk} (BRW).
More generally, an initial ancestor
is at the origin, and who forms the zeroth generation.  This parent
then produces children, the first generation, which are randomly
displaced from the parent according to some law.  Each of these
children behaves like an independent copy of the parent, their
children randomly displaced from their parent according to the same
law, and forming the second generation, and so on.  In our case,
each parent produces an infinite number of offspring, the
displacements from their parent given by $V=\{-\log y : y\in Z\}$,
where $Z$ is a point set with $PD(1)$ distribution.  We'll say that
$V$ has $LPD$ (logarithmic Poisson-Dirichlet) distribution from now
on.

Let $B_n$ be the minimum label of an individual at time
$n$, so that $B_n = -\log M_n$.  The first order behavior of 
the analog of $B_n$ (law of large numbers) for a general BRW
was determined in the 1970s by Biggins, Hammersley and Kingman
(see \cite{Big}).
In our case, Biggins' theorem \cite{Big} implies
$B_n \sim \frac{n}{\er}$ as $n\to\infty$  almost surely.
Thus,
$T\pfrac{\log 2}{\log p} \sim \er \log_2 p$ as $p\to\infty$ almost surely,
which justifies Conjecture \ref{Dpnormal1}.

Let $b_n = \text{median}(B_n)$.  The study of $B_n$ naturally breaks
into two parts: (i) global behavior: asymptotics for $b_n$,
and (ii) local behavior: the distribution of $B_n-b_n$.  A result of
McDiarmid \cite{Mc} can be used to prove $b_n=\frac{n}{\er}+O(\log n)$, and this
was sharpened by Addario-Berry and Ford \cite{ABF} to

\begin{thm}\label{bnsharp}  We have
$b_n = \frac{n}{\er} + \frac{3}{2\er}\log n + O(1).$
\end{thm}

\begin{cor} We have
$\text{median} \bigl(
T(\eps)\bigr) = \er \log (1/\eps) - \frac32 \log_2 (1/\eps)+O(1)$.
\end{cor}

This justifies part of Conjecture \ref{Dpnormal2}.  One ingredient in the proof 
is the following expectation identity.
Let $Z_n(t)$ be the number of generation $n$ individuals with position $\le t$,
and let $\zz^{(n)}$ be the set of positions of generation $n$ individuals.
If $\vv=(v_1,v_2,\ldots)$ has $PD(1)$ distribution, \eqref{Un} gives
\[
\EEE \sum_{j=1}^\infty v_j^s = \EEE \sum_{k=1}^\infty \bigl( (1-U_1)
\cdots (1-U_{k-1}) U_{k} \bigr)^s 
= \sum_{k=1}^\infty \frac{1}{(1+s)^k} = \frac{1}{s},
\]
since $\EEE U_i^s = \EEE (1-U_i)^s=1/(1+s)$.
By the branching property,
\[
\EEE \sum_{z_n\in \zz^{(n)}} \er^{-s z_{n}} =
\EEE \sum_{z_{n-1}\in \zz^{(n-1)}} \er^{-s z_{n-1}}
\EEE \sum_{z_1\in \zz^{(1)}} \er^{-s z_1} 
=\frac{1}{s} \EEE \sum_{z_{n-1}\in\zz^{(n-1)}} \er^{-s z_{n-1}}.
\]
By induction, the left side is $1/s^n$, so $\int_0^\infty \er^{-st}
d \EEE Z_n(t)=1/s^n$.  Therefore, $\EEE Z_n(t)=t^n/n!$.
Because  $t^n/n! \approx 1$ 
when $t=\frac{n}{\er} + \frac{1}{2\er}\log n + O(1)$, a naive
guess would be $b_n=\frac{n}{\er} + \frac{1}{2\er}\log n + O(1)$.
However, for reasons clearly explained in \cite{ABR} and executed in
\cite{ABF}, the leftmost
point in the $n$-th generation of a branching random walk has an
atypical ancestry with high probability.
Denote the locations of points in the ancestral line
of this leftmost point by $0,z_1,z_2,\ldots,z_n=B_n$ with $B_n$ close
to $b_n$.  Then usually
$z_j \ge \frac{j}{n} z_n - O(1)$ \ $(1\le j\le n/2)$.
A randomly chosen point $z_j\in \zz^{(n)}$ has this property
with probability of order $1/n$, so the expected number of such $z_j$
is in fact of order $t^n/(n\cdot n!)$, which is $\ge 1$ when $t\ge \frac{n}{\er} + \frac{3}{2\er}\log n + O(1)$.


We next discuss the local behavior of $B_n$.  Under very general
conditions on the BRW, it is known that $B_n-b_n$ is a tight
sequence.\footnote{A sequence $X_1, X_2, \ldots$ of random variables 
is tight if for every $\eps>0$ there is a number $M$ so that for all $j$,
$\PPP(|X_j|>M)\le \eps$.  In other words, the distribution of $X_j$
does not spread out as $j\to \infty$.}
The basic idea is that a single individual
will, with high probability, produce many offspring a few
generations later which are close by.
In our situation, tightness on the left for $H(p)$ is relatively easy to
prove unconditionally:

\begin{proof}[Proof of Theorem \ref{tightleft}]
The conclusion is trivial if $g(x^{1/2}) \le 3K$, so we will assume
that  $g(x^{1/2}) > 3K$. Let
$m=\lfloor g(x^{1/2})/K \rfloor$
so that $m\ge 3$.  Put $Q=x^{2^{-m}}$ and let
 $T$ be the set of primes $x^{1/2} < p\le x$ such that there is a prime
$q|(p-1)$ with $Q < q \le x^{1/4}$ and $H(q)\ge h(q)$.
For $p\in T$,
$$
H(p) \ge 1 + h(q) \ge h(Q) \ge h(x)-mK \ge h(p)-g(p),
$$
while by sieve methods (Theorem 4.2 of \cite{HR}), for large $x$
\[
| \{ x^{1/2} < p\le x: p\not\in T \}| \ll \frac{x}{\log x}
\prod_{\substack{Q < q \le x^{1/4} \\ H(q) \ge h(q)}} \(
  1-\frac{1}{q}\) \ll \frac{x}{\log x} 2^{-mc}.
\qedhere\]
\end{proof}

It is also known that under certain conditions on the
displacement law of the BRW (e.g. \cite{Ba}), the analog of $B_n-b_n$
converges in probability to a random variable as $n\to \infty$.
This is not known in our case.

\begin{conj}\label{conjX}
  $B_n-b_n \to X$ as $n\to\infty$ for a random variable $X$ with
  continuous distribution. 
\end{conj}

If $X$ exists, and the medians satisfy $b_{n+1}-b_n \to \er^{-1}$ as $n\to\infty$
(plausible in light of Theorem \eqref{bnsharp}), it
is easy to see that $X \deq -1/\er + \min_i ( z_i + X_i ),$
where $(z_1,z_2,\ldots)$ has $LPD$ distribution, $X_1, X_2, \ldots$
are independent copies of $X$, and $\deq$ means ``has the same distribution as''.  This follows by conditioning
on the positions of the first generation individuals (the points
$z_i$); that is, using
$B_n \deq \min_i ( z_i + B_{n-1}^{(i)} ),$
where $B_{n-1}^{(i)}$ are independent copies of $B_{n-1}$.  The solutions
$X$ of this recursive distributional equation are not known, however.

Unconditionally (whether $X$ exists or not), we prove that
$B_n-b_n$ has an exponentially decreasing left tail and doubly-exponentially 
decreasing right tail.
Consequently, if Conjecture \ref{conjX} holds, then all moments of $X$ exist.

\begin{thm}\label{tails}
(a) For any $c_1<\er$, we have
$$
\PPP\{B_n-b_n\le -x\} \ll_{c_1} \er^{-c_1 x} \qquad (n\ge 1, x\ge 0),
$$
and for any $c_2 > 2\er \log(2\er)$ and $\eta>0$, 
$$
\PPP\{B_n-b_n\le -x\} \gg_{c_2,\eta} \er^{-c_2 x} \qquad (n\ge 1, 0\le x\le (1/2\er-\eta)n);
$$
(b) for any $c_3<1$ there is a constant $c_4$, 
depending on $c_3$, so that
$$
\PPP \{ B_n \ge b_n + x \} \le \exp \( - \er^{c_3 (x-c_4)} \) \qquad
(n\ge 1, x\ge 0).
$$
\end{thm}

\begin{remark}
 By \eqref{Lamtail}, part (b) is nearly 
best possible; that is, the conclusion is false if $c_3>1$.
\end{remark}

The next two lemmas
hold for very general branching random walks.  A
notable feature is that they are \emph{local} results, and tightness
of $B_n-b_n$ can be proved without knowing anything about the growth
of $b_n$. We will use Theorem \ref{bnsharp} to 
prove the stronger tail estimates.

\begin{lem}\label{tight1}
For positive integers $m,n$ and positive real numbers $M$, $N$,
$$
\PPP \{ B_{m+n} \ge M+N \} \le \EEE [ (\PPP \{ B_n\ge N \} )^{Z_m(M)}].
$$
\end{lem}

\begin{proof}
Suppose $B_{m+n} \ge M+N$ and $Z_m(M) = k$.  For each of these $k$
individuals, all of their descendants in
generation $m+n$ are offset from their generation $m$ ancestor by at
least $N$.
\end{proof}

\begin{lem}\label{tight2}
Let $m,n$ be positive integers and let $M>0, \eps>0$ be real.
If $\EEE \{ (1-\eps)^{Z_m(M)} \} \le  \frac12$, then
$\PPP \{ B_n \le b_{n+m} - M \} \le \eps.$
In particular, the conclusion holds if $\PPP \{ Z_m(M) < 1/\eps \} \le \frac15$.
\end{lem}

\begin{proof}
Let $q$ be the $\eps$-quantile of $B_n$,
that is, $\PPP \{ B_n \le q \} = \eps$.  By Lemma \ref{tight1},
\[
\PPP \{ B_{m+n} \ge M+q \}
\le \EEE \left[ \( \PPP \{ B_n \ge q \} \)^{Z_m(M)} \right] \le \frac12.
\]
Therefore, $M+q \ge b_{m+n}$, and thus $\PPP \{ B_n \le b_{m+n} - M\} \le
\PPP \{ B_n \le q \} = \eps$.  To prove the second part,
assume that $\PPP \{ Z_m(M) < 1/\eps \} \le \frac15$.
Then
\[
\EEE \left\{ (1-\eps)^{Z_m(M)} \right\} \le \PPP 
\{ Z_m(m) < \tfrac{1}{\eps} \} + (1-\PPP\{Z_m(M)<\tfrac{1}{\eps}\})
(1-\eps)^{1/\eps}  \le \frac15 + \frac{4}{5\er} < \frac12.
\qedhere
\]
\end{proof}

\begin{lem}\label{Z1}
For real $t\ge 1$ and integer $k\ge 1$, we have
$\PPP \{ Z_1(t) \ge k \} \le (\er t/k)^{k-1}.$
\end{lem}

\begin{proof}
The conclusion is trivial if $k\le \er t$, so we suppose $k>\er t$.
Take  $s = \frac{k}{t} - 1$.  By \eqref{Un},
\begin{align*}
\PPP \{ Z_1(t)\ge k \} &\le \PPP \left\{ (1-U_1)\cdots (1-U_{k-1}) \ge 
\er^{-t} \right\} \\
&\le \er^{st} \int_0^1 \!\! \cdots \int_0^1 [(1-u_1) \cdots (1-u_{k-1})]^s
d u_1 \cdots d u_{k-1} 
= \frac{\er^{st}}{(1+s)^{k-1}}.
\qedhere
\end{align*}
\end{proof}

\begin{lem}\label{Zrklarge}
For all $r\ge 1$, $\theta>1$ and $\eps>0$, if $x$ is large then
$\PPP \{ Z_r(x) \ge \theta^x \} \le \exp \{ - (\theta-\eps)^x \}.$
\end{lem}

\begin{proof}
When $r=1$, this follows from Lemma \ref{Z1}.  Assume it to be true for some $r\ge 1$,
let $\theta$ and $\eps$ be given, and assume without loss of generality that
$\theta-\eps>1$.  The probability that $Z_r(x) \ge (\theta-\eps/3)^x$ is
$\le \exp \{-(\theta-\eps/2)^x\}$ for large $x$.  Now suppose 
$Z_r(x) = j <  (\theta-\eps/3)^x$ and $Z_{r+1}(x) \ge \theta^x$.
Let $m_i$ be the number of children of the $i$-th largest point in
 $\zz^{(r)}$ which are offset at most $x$ from their parent.  Let $\mathcal{I}$
be the set of indices with $m_i \ge 100x$.  
Note  that $m_1+\cdots+m_j\ge \theta^x$.
With $j, m_1, \ldots, m_j$ fixed, by Lemma \ref{Z1}
 the probability that $Z_{r+1}(x) \ge \theta^x$ is at most
$$
\prod_{i=1}^j \PPP \{ Z_1(x) \ge m_i \} \le \prod_{i\in\mathcal{I}} \er^{-2m_i}
\le \er^{-2(\th^x-100xj)} \le \exp \left\{ - \theta^x \right\}.
$$
As $m_i \le \er^x$,
the number of choices for $j,m_1,\ldots,m_j$ is at most
$\exp \{ (\theta-\eps/4)^x \}.$
For large $x$,
\[
\PPP \{ Z_r(x) \ge \theta^x \} \le \exp\{-(\theta-\eps/2)^x\} +
\exp\{  (\theta-\eps/4)^x - \theta^x \} \le \exp \{-(\theta-\eps)^x\}.
\qedhere
\]
\end{proof}

\begin{proof}[Proof of Theorem \ref{tails}]
Let $a>1/\er$ and $0<\eta<a\er/2$.
By \cite[Theorem 2]{BigChe}, for large $r$ we have
$\PPP \{ Z_r(ar) \le (a\er-\eta)^r \} \le \frac15$.
Let $r$ be so large that, in addition, $b_{n+r}\ge b_n+(1/\er-\eta)r$
for all $n$ ($r$ exists by Theorem \ref{bnsharp}).
Apply Lemma \ref{tight2} with $M=ar$, $m=r$,
$\eps=(a\er-\eta)^{-r}$.  For large integers $r$,
$$
\PPP \{ B_n \le b_n-(a-1/\er+\eta)r \} \le 
\PPP \{ B_n \le b_{n+r} - ar \} \le (a\er-\eta)^{-r}.
$$
The first estimate follows with $c_1=\frac{\log (a\er-\eta)}{(a-1/\er+\eta)}$.
Fix $a$, let $\eta\to 0$, then let $a\to 1/\er$, 
so that $c_1\to \er$.

For the second part of (a), take $\eta>0$ and $r$ as before (but fixed here),
 and let
$\del=(1/\er-\eta)r$, so that $b_{n+r}\ge b_n+\del$ for all $n$.
Since $\rho(u)=1-\log u$ for $1\le u\le 2$,
$\PPP(Z_1(\eps)\ge 1) = 1-\rho(\er^{\eps}) = \eps$ when $0\le \eps \le \log 2$.
Considering the ``leftmost child of the leftmost child of the \ldots
of the initial ancestor'' in the branching random
walk, we have
$\PPP \{ B_{kr} \le \del k/2 \} \ge \PPP \{ Z_1(\del/2r) \ge 1\}^{kr} 
= (\del/2r)^{kr}$ for every $k\ge 1$.  Hence,
$$
\PPP \{ B_n \le b_{n-kr} + \del k/2 \} \ge \PPP \{ B_{n-kr} \le b_{n-kr} \}
\PPP \{ B_{kr} \le \del k/2 \} \ge \frac12 \pfrac{\del}{2r}^{kr}.
$$
By assumption, $b_{n-kr}+\del k/2 \le b_n - \del k/2$.  Hence,
for $0\le k\le n/r$ we have
$$
\PPP \{ B_n \le b_n - \del k /2 \} \ge \frac12 \pfrac{\del}{2r}^{kr}.
$$
This gives the desired bound when $0\le x \le \frac{\delta n}{2r}$,
with $c_1'=\frac{2r}{\del} \log \frac{2r}{\delta}$.

To show part (b), we use induction on $n$ to show that
\be\label{tailsb} 
\PPP \{ B_n \ge b_n + x \} \le 2^{-\exp \{ c_3(x-c_5) \}}  
\ee 
for $n\ge 1$ and $x\ge 0$, where $c_5$ is sufficiently large.
Theorem \ref{tails} (c) then
follows with $c_4=c_5 - \frac{\log\log 2}{c_3}$.  As
\eqref{tailsb} is trivial for $0\le x \le c_5$, we may assume
$x\ge c_5$.   Let $r,\del$ be such that
$b_{n+r}-b_n \ge \del>0$ for all $n$ (the relationship between $r$ and
$\del$ is unimportant).
Let $A$ be a large integer, so that if $R=Ar$ and $\Delta=A\del$, then
$2\er^{2-\Delta} \le 1-c_3$.
Also suppose $c_3 \ge \frac12$.
For $1\le n\le R$, \eqref{Lamtail} implies
$\PPP \{ B_n \ge b_n+x \} \le \PPP \{ B_n \ge x \} = \rho_n(\er^x)
\le \exp \{ -\er^x \}$
if $c_5$ is large enough.  Suppose now that \eqref{tailsb} has been
proved for $1\le n\le m-1$, where $m-1\ge R$.
Define $\lam_j = \Delta + \frac{\log j - 1}{c_3}$ for $j\ge 1$
Let $j_0$ be the largest index $j$ with $\lam_j \le x+\Delta$.  
Let $z_1 \le z_2 \le \cdots$ be the points in $\zz^{(R)}$.  
For $1\le j\le j_0$, let $P_j$ be the event 
$\{z_i>\lam_i \; (i<j), z_j\le\lam_j\}$, and the $Q$ be the event
$\{z_i>\lam_i \; (1\le i\le j_0)\}$.  If $P_j$, then the generation $m$ points
descending from each of the $j$ points $z_1,\ldots,z_j$ are offset from their
generation $R$ ancestor by at least $b_m+x-\lam_j$. So
\[
\PPP \{ B_m \ge b_m + x \} \le \sum_{j=1}^{j_0} \PPP [P_j] \PPP \{ B_{m-R}
\ge b_m+x-\lam_j \}^j + \PPP [ Q ].
\]
Since $b_m \ge b_{m-R}+\Delta$, the induction hypothesis implies that the sum
on $j$ is
\[
\le  \sum_{j=1}^{j_0} \PPP [P_j] 2^{-j \exp\{c_3(x-c_5+\Delta-\lam_j)\}}
\le  \sum_{j=1}^{j_0} \PPP [P_j] 2^{-\exp\{c_3(x-c_5)+1\}}
\le 2^{-1-\exp\{c_3(x-c_5)\}}.
\]
Now suppose that $Q$ holds.  By the assumption on $A$,
$$
\sum_{j\le j_0} \er^{-z_j} \le \sum_{j=1}^{j_0} \er^{-\lam_j} \le
\er^{-\Delta+1/c_3} \sum_{j=1}^\infty j^{-1/c_3} \le \frac12.
$$
As $\lam_{j_0} \ge x+\Delta - (\lam_{j_0+1}-\lam_{j_0}) \ge x$,
$$
\sum_{\substack{z\in\zz^{(R)} \\ z\ge x }} \er^{-z} = 1 - 
\sum_{\substack{z\in\zz^{(R)} \\ z< x }} \er^{-z} \ge \frac12.
$$
Let $\eps=\frac13(\er-\er^{c_3})$ and $\th=\er-\eps$, so that $\er^{c_3}
<\theta-\eps < \theta < \er$.
For some integer $k\ge x$, there are $\ge \theta^k$ points of $\zz^{(R)}$ in
$[k-1,k)$, for otherwise
$$
\sum_{\substack{z\in\zz^{(R)} \\ z\ge x }} \er^{-z} \le \sum_{k\ge x}
\er \pfrac{\theta}{\er}^k < \frac12.
$$
By Lemma \ref{Zrklarge}, $\PPP[Q]\le \sum_k\PPP \{ Z_R(k) \ge \theta^k \} \le 
2\er^{ - (\theta-\eps)^x}$. 
This completes the proof of (b).
\end{proof}


\ms
\noindent
{\bf Acknowledgments.}  The authors wish to thank L.
Addario-Berry, J. Bertoin, C. Elsholtz, A. Granville,
A. Hildebrand, I. K\'atai, C. Pomerance,
Y. Sinai, and R. Song for helpful conversations.



\begin{thebibliography}{BFLSP}

\bibitem{ABF} L. Addario-Berry and K. Ford, {\it Poisson-Dirichlet
branching random walks}, preprint.

\bibitem{ABR} L. Addario-Berry and B. Reed, {\it Minima in branching
  random walks}, Ann. Prob. {\bf 37} (2009), 1044--1079.

\bibitem{Ba} M. Bachman, {\it Limit theorems for the minimal position
  in a branching random walk with independent logconcave
  displacements}, Adv. Appl. Prob. {\bf 32} (2000), 159--176.

\bibitem{BH} R. C. Baker and G. Harman, {\it Shifted primes without
  large prime factors}, Acta Arith. {\bf 83} (1998), 331--361.

\bibitem{BFLPS} W. D. Banks, J. Friedlander, F. Luca, F. Pappalardi
and I. E. Shparlinski, {\it Coincidences in the values of the Euler and
Carmichael functions}, Acta Arith. {\bf 122} (2006), 207--234.

\bibitem{BS}  W. D. Banks and I. E. Shparlinski, {\it On values taken
  by the largest prime factor of shifted primes},
 J. Aust. Math. Soc.  {\bf 82}  (2007),  no. 1, 133--147.

\bibitem{BKW} N. L. Bassily, I. K\'atai and M. Wijsmuller,
{\it Number of prime divisors of $\phi_k(n)$, where $\phi_k$ is the
$k$-fold iterate of $\phi$}, J. Number Theory {\bf 65} (1997), 226--239.

\bibitem{Bay} J. Bayless, {\it The Lucas-Pratt primality tree},
Math. Comp.  {\bf 77}  (2008), 495-502.

\bibitem{Big} J. D. Biggins, {\it The first- and last-birth problems
  for a multitype age-dependent branching process},
  Adv. Appl. Prob. {\bf 31} (1976), 446--459.

\bibitem{BigChe} J. D. Biggins, {\it Chernoff's theorem in the branching
randon walk}, J. Appl. Prob. {\bf 14} (1977), 630--636.

\bibitem{Bi} P. Billingsly, {\it
On the distribution of large prime divisors},
Collection of articles dedicated to the memory of Alfr\'ed R\'enyi, I.
Period. Math. Hungar. {\bf 2} (1972), 283--289.

\bibitem{Bo} E. Bombieri, {\it Le grand crible dans la th\'eorie
    analytique des nombres}, 2nd ed. (French.  English summary)
Ast\'erisque No. 18. Soci\'et\'e Math\'ematique de France, Paris, 1987.  

\bibitem{Bor} C. W. Borchardt, {\it \"Uber eine Interpolationsformel 
f\"ur eine Art Symmetrischer Functionen und \"uber Deren Anwendung}, 
Math. Abh. der Akademie der Wissenschaften zu Berlin (1860), 1--20.

\bibitem{Che} M. A. Cherepnev, {\it Some properties of large 
prime divisors
of numbers of the form $p-1$}, Mat. Zametki {\bf 80} (2006), 920--925.
(Russian).  English translation in Math. Notes {\bf 80} (2006), 863--867.

\bibitem{CP} R. Crandall and C. Pomerance, {\it Prime numbers: a
  computational perspective}, 2$^{\text{nd}}$ ed., Springer-Verlag, 2005.

\bibitem{Da} H. Davenport, {\it Multiplicative number theory}, 3rd ed.,
Graduate Texts in Mathematics vol. 74, Springer-Verlag, New York, 2000.

\bibitem{DG} P. Donnelly and G. Grimmett, {\it On the asymptotic
    distribution of large prime factors}, J. London Math. Soc.
(2) {\bf 47} (1993), 395--404.

\bibitem{Er35}  P. Erd\H os, {\it On the normal number of prime
factors of $p-1$ and some related problems concerning Euler's
$\phi$-function},  Quart. J. Math. Oxford {\bf 6} (1935), 205-213.

\bibitem{EGPS} P. Erd\H os, A. Granville, C. Pomerance, and C. Spiro,
{\it On the normal behavior of the iterates of some arithmetic functions},
in Analytic Number Theory, Proceedings of a conference in honor of Paul T.
Bateman, Birkh\" auser, Boston, 1990, 165--204.

\bibitem{EP} P. Erd\H os and C. Pomerance,
{\it On the normal number of prime factors of $\phi(n)$},
 Number theory (Winnipeg, Man., 1983).  Rocky Mountain J. Math.  {\bf
   15} (1985),  no. 2, 343--352.

\bibitem{FL} K. Ford and F. Luca, {\it The number of solutions of
$\lambda(x)=n$}, to appear in INTEGERS, special volume in honor of Melvyn
Nathanson and Carl Pomerance.

\bibitem{phisig} K. Ford, F. Luca and C. Pomerance, {\it Common
values of the arithmetic functions $\phi$ and $\sigma$}, Bull. London
Math. Soc. {\bf 42} (2010), 478--488.

\bibitem{GT} B. Green and T. Tao,
{\it The primes contain arbitrarily long arithmetic
  progressions}, Ann. Math. {\bf 167} (2008), 481--547.

\bibitem{HR} H. Halberstam and H.-E. Richert, {\it Sieve methods\/},
Academic Press, London, 1974.

\bibitem{HaTe} R. R. Hall and G. Tenenbaum, {\it Divisors},
Cambridge Tracts in Math., 1988.

\bibitem{HT} A.~Hildebrand and G.~Tenenbaum,
{\it Integers without large prime factors},
J.\ de Th{\'e}orie des Nombres de Bordeaux,
{\bf 5} (1993), 411--484.

\bibitem{Ka1} I. K\'atai, {\it Some problems on the iteration of multiplicative
number-theoretical functions}, Acta Math. Acad. Scient. Hung. {\bf
  19} (1968), 441--450.

\bibitem{Lam} Y. Lamzouri, {\it Smooth values of iterates of the Euler
phi-function}, Canad. J. Math. {\bf 59} (2007), 127--147.

\bibitem{LP} F. Luca and C. Pomerance, {\it Irreducible radical
  extensions and Euler-function chains}, in ``Combinatorial number
  theory'', 351--361, de Gruyter, Berlin (2007).

\bibitem{MP} G. Martin and C. Pomerance, {\it The iterated Carmichael
  $\lambda$-function and the number of cycles of the power generator},
  Acta Arith. {\bf 188} (2005), 305--335.

\bibitem{Mc} C. McDiarmid, {\it Minimal positions in a branching
  random walk}, Ann. Appl. Prob. {\bf 5} (1995), no. 1, 128--139.

\bibitem{Pom87}  C. Pomerance, {\it Very short primality proofs},
Math. Comp. {\bf 48} (1987), 315--322.

\bibitem{Pr} V. Pratt, {\it Every prime has a succinct certificate},
SIAM J. Comput. {\bf 4} (1975), no. 3, 214--220.


\end{thebibliography}
\end{document}